\documentclass[final,onefignum,onetabnum]{siamart220329}
\usepackage{lipsum}
\usepackage{amsfonts}
\usepackage{amssymb}
\usepackage{graphicx}
\usepackage{epstopdf}
\usepackage{algorithmic}
\usepackage{enumitem}
\usepackage{booktabs}
\usepackage{multirow}
\usepackage{amsopn}
\usepackage{bm}
\ifpdf
\DeclareGraphicsExtensions{.eps,.pdf,.png,.jpg}
\else
\DeclareGraphicsExtensions{.eps}
\fi

\newsiamremark{remark}{Remark}
\newsiamthm{exper}{Experiment} 
  
\DeclareMathOperator{\diag}{diag}
\DeclareMathOperator{\spans}{span}

\DeclareMathOperator{\new}{new}

\newcommand{\ellc}{\ell_{\mathrm{c}}}

\newcommand{\XX}{\mathcal{X}}

\headers{A NEW JACOBI--DAVIDSON VARIANT FOR HERMITIAN EIGENPROBLEMS}{JINZHI HUANG}

\title{JD-V: A new variant of Jacobi--Davidson method
	for large Hermitian eigenproblems\thanks{Submitted to the editors DATE. 
		\funding{This work is supported by the Youth Fund of the 
			National Science Foundation of China (Grant No. 12301485) 
			and the Jiangsu Province Youth Science and
			Technology Talent Support Program (Grant No. JSTJ-2025-828)}}}

\author{Jinzhi Huang\thanks{School of Mathematical Sciences, Soochow 
		University, 215006 Suzhou, China (\url{jzhuang21@suda.edu.cn}).}}
\begin{document}
%\linenumbers
\maketitle

\begin{abstract}
A novel variant of the Jacobi-Davidson (JD) type method for Hermitian 
eigenvalue problems, designated as JD-V, is proposed based on a newly 
designed correction equation, whose solution is shown to be nearly as 
effective as that of the standard correction 
equation for subspace expansion. 
Rigorous convergence analysis of MINRES for solving these  
equations reveals that the inner iterations of JD-V are significantly 
more efficient than those of the standard JD method when highly clustered 
eigenvalues are of interest. 
A thick-restart JD-V algorithm with deflation and purgation is developed 
to compute several eigenpairs of a a large-scale Hermitian matrix. 
Numerical experiments confirm the theoretical results and demonstrate the 
considerable superiority of JD-V over standard JD in overall efficiency. 
\end{abstract}
 


\section{Introduction}\label{sec:1}
Let $A\in\mathbb{C}^{n\times n}$ be a large-scale, possibly 
sparse, and Hermitian matrix, i.e., $A^*=A$, where the superscript 
$*$ denotes the conjugate transpose.
Consider its eigenvalue decomposition: 
\begin{equation}\label{eigA}
  A=X\Lambda X^*,
\end{equation}
where $X=[x_1,\dots,x_n]\in\mathbb{C}^{n\times n}$ is unitary, and  
$\Lambda=\diag\{\lambda_1,\dots,\lambda_n\}\in\mathbb{R}^{n\times n}$.  
Such $(\lambda_i,x_i)$ are called eigenpairs of $A$, with 
$\lambda_i\in\mathbb{R}$ the eigenvalues, and $x_i\in\mathbb{C}^{n}$ 
 the eigenvectors, $i=1,\dots,n$. 
Given a target $\tau\in\mathbb{R}$, assume that 
\begin{equation}\label{order}
	|\lambda_1-\tau|<\dots<|\lambda_{\ell}-\tau|
	<|\lambda_{\ell+1}-\tau|\leq\dots\leq|\lambda_{n}-\tau|.
\end{equation}
We aim to compute the $\ell$ eigenpairs $(\lambda_1,x_1),\dots,
(\lambda_\ell,x_\ell)$ with the eigenvalues closest to $\tau$. 
In particular, if $\tau$ lies within the spectrum $(\lambda_{\min}, \lambda_{\max})$ 
of $A$ and is  close to neither the smallest eigenvalue $\lambda_{\min}$ nor the largest one $\lambda_{\max}$,  
then $(\lambda_i,x_i)$ are classified as interior eigenpairs; otherwise, they are extreme, i.e., smallest or largest, ones. 
  
Several methods \cite{bai2000templates,parlett1998symmetric,
	stewart2001matrix} are available for solving 
problems of this kind. Among them, the Jacobi-Davidson (JD) method 
\cite{sleijpen1996jacobi,sleijpen2000jacobi,sleijpen1998efficient} is 
a popular one. 
This is an inner-outer iterative method: the outer iterations compute 
the desired eigenpairs from the search subspace using one of the standard, 
harmonic, refined, and refined harmonic extraction approaches 
\cite{jia1997refined,jia2002refinedh,stewart2001matrix}, while the inner 
iterations employ, e.g., a Krylov subspace solver 
\cite{barrett1994templates,golub2012matrix}, to approximately 
solve an $n$-by-$n$ Hermitian (typically indefinite) linear system, 
called correction equation, to expand the search subspace. 

The inner iterations dominate the overall efficiency of the JD method. 
Fortunately, the correction equations do not need to be solved with high 
accuracy. 
Notay \cite{notay2002combination} analyzes the simplified JD method (JDCG) 
for Hermitian eigenvalue problems, in which the correction equations 
are solved by the preconditioned conjugate gradient (PCG) method 
\cite{greenbaum1997iterative}, and the solution is directly added to the 
current approximate eigenvector to form a new one. 
He monitors the outer convergence via the reduction of inner residual 
norms designs stopping criteria for the inner iterations. 
Hochstenbach and Notay \cite{hochstenbach2009controlling} extend these 
results to Hermitian generalized eigenvalue problems. 
Stathopoulos and McCombs \cite{stathopoulos2010primme} adapt the stopping 
criteria from \cite{hochstenbach2009controlling,notay2002combination} to
the inner iterations of their general JD methods with subspace acceleration for Hermitian eigenvalue problems. 
%Instead of PCG, they employ the quasi-minimal residual (QMR) algorithm 
%\cite{stathopoulos2007nearly} to solve the correction equations. 
For general JD methods, Jia and Li 
\cite{jia2014inner,jia2015harmonic} prove that solving the correction 
equations to {\em low or modest} accuracy with the relative 
error of approximate solutions dropping below 
$\varepsilon_{\mathrm{in}}\in[10^{-4},10^{-3}]$ suffices to make the 
outer iterations of the resulting JD behave as if all the 
correction equations were solved exactly. 
Specifically, if the exact JD method converges rapidly in the sense of 
requiring only about ten outer iterations, then 
$\varepsilon_{\mathrm{in}}\approx10^{-4}$ suffices for the resulting 
inexact JD method to converge in almost the same number of outer iterations. 
If the exact JD method converges slowly and uses many outer iterations, 
then $\varepsilon_{\mathrm{in}}\in[10^{-3},10^{-2}]$ enables the inexact 
JD method to mimic the convergence behavior of its exact counterpart well. 
Based on those, they propose practical stopping criteria 
for the inner solvers and improve the overall efficiency of JD methods 
substantially. 
The author and Jia \cite{huang2019inner} extend the results in 
\cite{jia2014inner,jia2015harmonic} to JD-type singular value decomposition 
(SVD) methods.
 
For Hermitian linear systems, the method of choice among Krylov subspace 
solvers is the minimal residual (MINRES) method \cite[Section 2.4]{greenbaum1997iterative}. 
Its convergence rate is uniquely determined by the 
eigenvalues of the coefficient matrix. 
Specifically, MINRES may converge slowly when the coefficient matrix 
has relatively small eigenvalues in magnitude; if these small eigenvalues 
are discarded, MINRES will converge much faster. 
As will be shown, when  $\lambda_1$ is close to $\tau$, if the 
orthogonal projector in the correction equation of JD 
is replaced by a new one that is properly constructed from approximate eigenvectors associated with all the clustered eigenvalues near $\tau$ rather than only $\lambda_1$, then the coefficient matrix of the new correction equation has no 
eigenvalues of small magnitude. In this case, MINRES can converge much faster. 
Remarkably, extensive numerical experiments have shown that as JD method proceeds, the search subspace yields 
increasingly accurate approximate eigenvectors corresponding 
to those clustered eigenvalues. 
In this paper, inspired by the work of the author and Jia 
\cite{huang2026preconditioning}, we will fully utilize this to derive new correction equations and propose a new
variant of JD, which we call JD-V. 

We show that, under certain conditions, the solutions to the proposed and 
standard correction equations are equally effective in expanding the 
search subspace, ensuring that the outer convergence of JD‑V mimics that of JD.  
Then, we present a rigorous convergence analysis of MINRES in a unified 
form for both equations, showing that MINRES converges significantly 
faster for the new equation than for the standard one, provided that the eigenvalue of interest is tightly clustered with others. 
This is particularly common when interior eigenpairs are sought.  
Additionally, we explain why the correction equation proposed in 
\cite{genseberger1999alternative} is less effective than those in JD and 
JD‑V, and why the corresponding JD variant requires more outer iterations 
to converge. For practical purpose, we establish quantitative criteria 
and develop an adaptive approach to dynamically form the new correction 
equation at each outer iteration. 
Moreover, we adapt the new thick-restart strategy proposed in 
\cite{huang2026preconditioning} for large SVD computations to our context, 
and develop a thick-restart JD-V algorithm that incorporates deflation 
and purgation techniques to compute  $\ell$  eigenpairs of $A$.  
   
Specifically, we take the standard extraction-based JD method as an 
example to derive the new correction equations. The 
theoretical results and analysis are also adaptable to the harmonic, refined,  
and refined harmonic extractions-based JD methods \cite{stathopoulos2007nearly,
	stathopoulos2007nearly2,stathopoulos2010primme}; 
so are their corresponding thick-restart algorithms.  
Moreover, any preconditioner designed for the standard correction 
equations in JD can be applied directly to their counterparts in JD-V. 

The rest of this paper is organized as follows. 
Section~\ref{sec:2} reviews the basic standard 
JD method for computing one eigenpair of $A$, where the shift  
in every correction equation is fixed to the target $\tau$ to simplify 
the analysis and presentation.
Section~\ref{sec:3} derives the new 
correction equation, analyzes the convergence of MINRES when applied 
to it, and discusses its adaptive formation in practical computations.
Section~\ref{sec:4} develops the thick-restart JD-V algorithm with 
adaptively varying shifts, equipped with deflation and purgation 
techniques, to compute several eigenpairs of $A$.  
Section~\ref{sec:5} reports the numerical experiments, and 
section~\ref{sec:6} concludes the paper.
 
\section{The basic JD method}\label{sec:2}
We review the basic standard extraction-based JD method \cite{sleijpen2000jacobi} for computing the eigenpair 
$(\lambda_1,x_1)$ of $A$.

At the $k$th outer iteration, assume that a $k$-dimensional search  
subspace $\XX\subset\mathbb{C}^n$ is available. The standard 
extraction approach seeks scalars $\theta_i\in\mathbb{R}$ 
and unit-length vectors $\tilde x_i\in\XX$ satisfying the 
following  
standard orthogonal requirement:
\begin{equation}\label{extraction}
	\qquad  r_i:=A\tilde x_i-\theta_i\tilde x_i\perp \XX,
	\qquad  i=1,\dots,k. 
\end{equation}
Such $(\theta_i,\tilde x_i)$ are called the Ritz pairs of $A$ with 
respect to $\XX$, where $\theta_i$ and $\tilde x_i$ are the Ritz 
values and Ritz vectors, respectively,  
and $r_i$ are the associated residuals. 
 
Let the columns of $\widetilde X\in\mathbb{C}^{n\times k}$ form an 
orthonormal basis for $\XX$.
Write $\tilde x_i=\widetilde Xs_i$ for some
$s_i\in\mathbb{C}^{k}$, $i=1,\dots,k$. Then \eqref{extraction} amounts to 
\begin{equation}\label{projection}
 \qquad Hs_i=\theta_is_i,\qquad i=1,\dots,k, 
\end{equation}  
meaning that $(\theta_i,s_i),i=1,\dots,k$ are eigenpairs of $H=\widetilde X^*A\widetilde X$. Therefore, in practice, the standard  
JD method computes the eigendecomposition of 
$H$, orders the eigenvalues by  $|\theta_1-\tau|\leq|\theta_2-\tau|\leq\dots\leq|\theta_k-\tau|$,   
and takes the Ritz pairs
\begin{equation}\label{Ritzpair}
	\qquad  
	(\theta_i,\tilde x_i)=(\theta_i,\widetilde X s_i), 
	\qquad 
	i=1,\dots,k.
\end{equation}
Among them, take $(\theta_1,\tilde x_1)$ as the approximation 
to the desired eigenpair $(\lambda_1,x_1)$. 
It is considered to have converged if its residual 
$r_1$, defined by \eqref{extraction}, satisfies 
\begin{equation}\label{converg}
	\|r_1\|\leq\|A\|_1\cdot\varepsilon_{\mathrm{out}} 
\end{equation}
for a user-prescribed outer stopping tolerance $\varepsilon_{\mathrm{out}}>0$, 
where $\|A\|_1$ is the $1$-norm of $A$. 
In this case, we terminate the JD method. 

If $(\theta_1,\tilde x_1)$ has not yet converged, the JD method
approximately solves a large-scale Hermitian and generally 
indefinite linear correction equation 
\begin{equation}\label{correction}
	\qquad 	
	(I-\tilde x_1\tilde x_1^*)(A-\tau I)
	(I-\tilde x_1\tilde x_1^*)t=-r_1
	\quad\mbox{with}\quad
	t\perp \tilde x_1,  
\end{equation}
using a Krylov subspace-type iterative method, e.g., the commonly 
used MINRES method, 
to {\em low or modest accuracy} \cite{jia2014inner,jia2015harmonic}; see \eqref{innstop}.  
The approximate solution $\tilde t$ is then used  to expand the 
search subspace: % $\XX_{\new}=\spans\{\XX,\tilde t\}$.
\begin{equation}\label{Xnew} 
	\XX_{\new}=\spans\{\XX,\tilde t\}.
\end{equation} 
Specifically, the orthonormal basis matrix for $\XX_{\new}$ is updated by 
\begin{equation}\label{expansion}
\widetilde X_{\new}=[\widetilde X,x_+]
\qquad\mbox{with}\qquad	
x_{+}=\frac{(I-\widetilde X\widetilde X^*)
		\tilde t}{\|(I-\widetilde X\widetilde X^*)\tilde t\|}.
\end{equation}  
We then seek a new Ritz approximation to the desired 
 $(\lambda_1, x_1)$ from $\XX_{\new}$ as described above. 
The basic JD method proceeds in this manner until $(\theta_1,\tilde x_1)$ converges.   
 
Solving the correction equation \eqref{correction} dominates the 
overall computational cost of the JD method. 
However, as observed numerically and  later confirmed theoretically, 
MINRES may converge very slowly for \eqref{correction} when 
$\lambda_1$ is tightly clustered with other eigenvalues of $A$ 
(see Remark~\ref{remark1}), which corresponds to the case where the 
coefficient matrix $\widetilde B_1$ of \eqref{correction} has 
a cluster of eigenvalues of small magnitude.  
Fortunately, as JD proceeds, the search subspace $\XX$ naturally 
captures increasingly accurate approximations to the eigenvectors 
associated with the clustered eigenvalues near $\tau$. 
These approximations, in turn, provide crucial information regarding 
the eigenvectors corresponding to the smallest-in-magnitude eigenvalues 
of $\widetilde B_1$. 
Motivated by this, we shall formulate a new correction equation that 
is significantly more amenable to MINRES, while ensuring that its 
solution expands $\XX$ as effectively as that of \eqref{correction}.  
 
\section{A new correction equation-based JD-V type method}\label{sec:3}
Suppose that there are $m$ clustered eigenvalues $\lambda_1,\dots,\lambda_m$ of $A$ near $\tau$:
\begin{equation}\label{ordernew}
	|\lambda_1-\tau|\lesssim\dots\lesssim|\lambda_{m}-\tau| 
	\ll |\lambda_{m+1}-\tau|\leq\dots\leq|\lambda_{n}-\tau|.
\end{equation}
From \eqref{eigA}, denote the corresponding $m$-dimensional block eigenpair by
\begin{equation}\label{block}
	(\Lambda_m,X_{m}) = 
	\left(\diag\{\lambda_1,\dots,\lambda_m\},
	[x_1,\dots,x_m] \right), 
\end{equation}
which satisfies $AX_m=X_m\Lambda_m$ and $\Lambda_m=X_m^*AX_m$.  
At the $k$th outer iteration with $k\geq m$, suppose that 
the Ritz pairs $(\theta_i,\tilde x_i), i=1,\dots, m$ computed in 
\eqref{Ritzpair} are sufficiently accurate to represent the exact 
eigenpairs $(\lambda_i,x_i),i=1,\dots,m$. 
Then  
\begin{equation}\label{blockRitz}
	(\Theta_m,\widetilde X_{m} )
	=\left(\diag\{\theta_1,\dots,\theta_m\} ,[\tilde x_1,\dots,\tilde x_m]\right) 
\end{equation} 
serves as a reliable approximation to $(\Lambda_m,X_m)$. 
Additionally, it is easily verified from \eqref{projection} that
 $\widetilde X_m$ is column-orthonormal, and $\Theta_m = \widetilde X_m^*A\widetilde X_m$,  
%\begin{equation}\label{Theta}
%	\Theta_m = \widetilde X_m^*A\widetilde X_m,   
%\end{equation}
meaning that $\Theta_m$ is a block Rayleigh quotient of $A$ with respect to $\widetilde X_m$.   
  
\subsection{The new correction equation}\label{subsec:1} 
Partition $\widetilde X_m=[\tilde x_1,\widetilde X_{2:m}]$ with $\widetilde X_{2:m}=[\tilde x_2,\dots,\tilde x_m]$. Since the columns of $\widetilde X_m$ are  orthonormal, 
we have  
\begin{equation*} 
	(I-\widetilde X_m\widetilde X_m^*)(I-\tilde x_1\tilde x_1^*)
	=I-\widetilde X_m\widetilde X_m^* 
	\quad\mbox{and}\quad
	I-\tilde x_1\tilde x_1^*=
	(I-\widetilde X_m\widetilde X_m^*) 
	+\widetilde X_{2:m}\widetilde X_{2:m}^*.
\end{equation*}
Besides, condition \eqref{extraction} imposes 
$ r_1\perp\widetilde X_m$, so that $(I-\widetilde X_m\widetilde X_m^*)r_1=r_1$. 
As a consequence, premultiplying both sides of \eqref{correction} by $I-\widetilde X_m\widetilde X_m^*$ yields 
\begin{eqnarray}
	(I-\widetilde X_m\widetilde X_m^*)(A-\tau I)
	(I-\widetilde X_m\widetilde X_m^*)t
	&=&-r_1-(I-\widetilde X_m\widetilde X_m^*)(A-\tau I)\widetilde X_{2:m}\widetilde X_{2:m}^*t  \nonumber  \\
	%&=&-r_1-(A\widetilde X_{2:m}-\widetilde X_m\widetilde X_m^*A\widetilde X_{2:m})\widetilde X_{2:m}^*t \nonumber \\
	&=&-r_1-(A\widetilde X_{2:m}-\widetilde X_{2:m}\Theta_{2:m})\widetilde X_{2:m}^*t \nonumber \\ 
	&=& -r_1-R_{2:m}\widetilde X_{2:m}^*t. \label{derive1} %.\label{rtail}
\end{eqnarray} 
Here, $\Theta_{2:m}= \widetilde X_{2:m}^*A\widetilde X_{2:m}=\diag\{\theta_2,\dots,\theta_m\}$, and  
\begin{equation*}%\label{R2}
R_{2:m} = A\widetilde X_{2:m}-\widetilde X_{2:m}\Theta_{2:m}
=[r_2,\dots,r_m] 
\end{equation*} 
is the block residual associated with the block Ritz pair $(\Theta_{2:m},\widetilde X_{2:m})$, where the column $r_i$ is the residual of $(\theta_i,\tilde x_i)$, $i=2,\dots,m$; see \eqref{extraction}.  
%Then \eqref{rtail} can be rewritten as
%\begin{equation}%\label{derive1}
%	(I-\widetilde X_m\widetilde X_m^*)(A-\tau I)
%	(I-\widetilde X_m\widetilde X_m^*)t=-r_1-R_{2:m}\widetilde X_{2:m}^*t.
%\end{equation} 

As is standard in the literature, $\|t\|$ is asymptotically proportional to the error in $\tilde x_1$, while each $\|r_i\|$ is of the same order as the error in $\tilde x_i$, $i=1,\dots,m$.  We have 
\begin{equation*}%\label{R_2}
	\|R_{2:m}\widetilde X_{2:m}^*t\|\leq\|R_{2:m}\|\|t\|\approx \mathcal{O}(\|r_1\|^2) 
\end{equation*}
provided that 
$\|r_i\|$, $i=2,\dots,m$ are comparable to $\|r_1\|$,  
implying that $(\theta_i,\tilde x_i)$, $i=2,\dots,m$ have similar accuracy to $(\theta_1,\tilde x_1)$. 
Alternatively, as long as $(\theta_i,\tilde x_i)$, $i=2,\dots,m$ have sufficient accuracy---albeit inferior to $(\theta_1,\tilde x_1)$ as the latter approaches convergence---the asymptotic relation  
\begin{equation}\label{R_3}
	\|R_{2:m}\widetilde X_{2:m}^*t\|\leq\|R_{2:m}\|\|t\|\approx \|R_{2:m}\|\cdot \mathcal{O}(\|r_1\|)   
\end{equation} 
remains valid. In these cases, the approximation 
$r_1+R_{2:m}\widetilde X_{2:m}^*t\approx r_1$ holds, provided that $\|R_{2:m}\|$ is comparable to $\|r_1\|$ or sufficiently small. 
 
Neglecting the higher-order term $R_{2:m}\widetilde X_{2:m}^*t$ in \eqref{derive1} leads to %the simplified equation 
\begin{equation*}%\label{inter}
	(I-\widetilde X_m\widetilde X_m^*)(A-\tau I)
	(I-\widetilde X_m\widetilde X_m^*)t=-r_1
	\quad\mbox{with}\quad
	t\perp\tilde x_1.
\end{equation*} 
%whose solution is constrained by $t\perp\tilde x_1$, inherited from \eqref{correction}. 
Clearly, the solution 
is non-unique, and can be written as 
$t_{\star}+\spans\{\widetilde X_{2:m}\}$ for a specific 
$t_{\star}\perp\tilde x_1$. Nevertheless, by virtue of the subspace expansion mechanism described in section~\ref{sec:2}, any such solution yields the same expansion vector $x_+$ in \eqref{expansion}, thereby expanding $\XX$ equivalently. 
To guarantee uniqueness, we impose the additional constraint 
$t\perp \spans\{\widetilde X_{2:m}\}$, which, combined with $t\perp \tilde x_1$, dictates that $t\perp\widetilde X_m$.   
We thus arrive at the new correction equation:  
\begin{equation}\label{correctionm} 
	(I-\widetilde X_m\widetilde X_m^*)(A-\tau I)
	(I-\widetilde X_m\widetilde X_m^*)t=-r_1 
	\quad\mbox{with}\quad
	t\perp \widetilde X_m,
\end{equation}
whose unique solution shares the similar subspace-expanding efficacy to 
that of \eqref{correction}. 

Therefore, at the $k$th step, if $(\theta_1,\tilde x_1)$ fails to converge 
while $(\theta_i,\tilde x_i)$, $i=2,\dots,m$ are sufficiently accurate, 
we apply MINRES to solve \eqref{correctionm} to a low or modest accuracy, 
and use the approximate solution $\tilde t$ to expand the search subspace 
via \eqref{Xnew}--\eqref{expansion}. 
The resulting approach is referred to as the JD-V method.
 
\subsection{Convergence analysis on MINRES for \eqref{correctionm}}\label{subsec:2} 
Observe that the coefficient matrix $\widetilde B_m$ of \eqref{correctionm} is the restriction of $A-\tau I$ to the 
orthogonal complement $\widetilde X_m^{\perp}$. 
Let the columns of $\widetilde Y_m$ form an orthonormal basis for $\widetilde X_m^{\perp}$. 
Then we have 
\begin{equation}\label{Mm}
	\widetilde B_m=\widetilde Y_m\widetilde B_m^{\prime}\widetilde Y_m^*
	\qquad\mbox{with}\qquad
	\widetilde B_m^{\prime}=\widetilde Y_m^*(A-\tau I)\widetilde Y_m. 
\end{equation}
 
Denote by $\tilde t_j$ the approximate solution to \eqref{correctionm} generated at the $j$th MINRES iteration with the starting vector 
$\tilde t_0\in\widetilde X_m^{\perp}$, $j=1,2,\dots,n-m$.  
The inner residual is 
\begin{equation*}%\label{innres}
	r_{\mathrm{in}, j}=\widetilde B_m\tilde t_j+r_1.
\end{equation*}  
Before proceeding, it is useful to cast \eqref{correctionm} 
into an equivalent but simpler form. 
Since the derivation is analogous to that of \cite[Theorem~3.3]{huang2026preconditioning}, the detailed proof is omitted.  

\begin{theorem}\label{thm1}
	The solution to \eqref{correctionm} takes the form 
	$t=\widetilde Y_m t^{\prime}$, where 
	$t^{\prime}$ 
	satisfies 
	\begin{equation}\label{correction2}
		\widetilde B_m^{\prime}t^{\prime}=-r^{\prime}_1 = - \widetilde Y_m^*r_1.
	\end{equation} 
Let $\tilde t^{\prime}_j$ be the approximate solution to \eqref{correction2} generated at the $j$th MINRES iteration using the starting guess $\tilde t^{\prime}_0=\widetilde Y_m^*\tilde t_0$,   $j=1,2,\dots,n-m$. Then $\tilde t_j=\widetilde Y_m\tilde t^{\prime}_j$,  
and 
\begin{equation*} 
		\qquad \qquad  	
		\|r_{\mathrm{in}, j}\|=\|r_{\mathrm{in}, j}^{\prime}\|=\|\widetilde B_m^{\prime}\tilde t_j^{\prime}+r^{\prime}_1\|,
		\qquad j=1,2,\dots,n-m.
	\end{equation*}  
\end{theorem}
 
Theorem~\ref{thm1} establishes the convergence equivalence between MINRES applied 
to \eqref{correctionm} and \eqref{correction2}.  
For the sake of brevity, there is no loss of generality in assuming that 
the starting vectors $\tilde t_0=\bm{0}$ and $\tilde t_0^{\prime}=\bm{0}$. 
Then the initial residuals reduce to 
$r_{\mathrm{in},0} = r_1^{\prime}$ and $r_{\mathrm{in},0}^{\prime} = r_1^{\prime}$.
%\begin{equation*}%\label{initial}
%r_{\mathrm{in},0} = r_1^{\prime}
%\quad\mbox{and}\quad
%r_{\mathrm{in},0}^{\prime} = r_1^{\prime}.
%\end{equation*} 
Let $\tilde\sigma_1,\dots,\tilde\sigma_{n-m}$ be the eigenvalues of  $\widetilde B_m^{\prime}$. 
It is known from  
\cite[p.~51]{greenbaum1997iterative} that the relative residual norm 
of $\tilde t_j^{\prime}$ satisfies %the min-max bound 
\begin{equation}\label{minmax}
	\frac{\|r_{\mathrm{in},j}^{\prime}\|}{\|r^{\prime}_1\|}
	\leq\min_{\pi\in\Pi_{j}}\max_{1\leq i\leq n-m} |\pi(\tilde\sigma_i)|,
\end{equation}
where $\Pi_j$ denotes the class of polynomials $\pi$ of degree at most 
$j$, normalized such that $\pi(0)=1$. 
Combining Theorem~\ref{thm1} and \eqref{minmax}, we introduce the following result. 

\begin{theorem}[{\cite[Chapter~3]{greenbaum1997iterative}}]
	\label{thm2}
(\romannumeral1) If all the eigenvalues $\tilde\sigma_1,\dots,\tilde\sigma_{n-m}$ 
of $\widetilde B_m^{\prime}$ lie in either the interval 
$[\alpha,\beta]$ or $[-\beta,-\alpha]$ with $\beta>\alpha>0$, then 
\begin{equation}\label{est1}
	\frac{\|r_{\mathrm{in},j}\|}{\|r_1\|}\leq 2\left(1-\frac{2}{1+\sqrt{\beta}/\sqrt{\alpha}}\right)^{j};
\end{equation}
(\romannumeral 2) If those eigenvalues are located within 
$[-\beta_1,-\alpha_1]\cup[\alpha_2,\beta_2]$, where the positive constants $\alpha_i,\beta_i$, $i=1,2$ satisfy  
$\beta_1-\alpha_1=\beta_2-\alpha_2$, then 
\begin{equation}\label{est2}
	\frac{\|r_{\mathrm{in},j}\|}{\|r_1\|}\leq 2\left(
	1-\frac{2}{1+\sqrt{\beta_1\beta_2}/\sqrt{\alpha_1\alpha_2}}
	\right)^{\lbrack\frac{j}{2}\rbrack},
\end{equation}
where $\lbrack\frac{j}{2}\rbrack$ represents the integer part of $\frac{j}{2}$.
\end{theorem}

Theorem~\ref{thm2} indicates that the convergence of MINRES for \eqref{correctionm}
becomes faster as the ratio $\beta/\alpha$ or 
$(\beta_1\beta_2)/(\alpha_1\alpha_2)$ approaches one.  
In particularly, when \eqref{correctionm} is positive or negative 
definite, the convergence rate in \eqref{est1} reduces to  
$\frac{\sqrt{\kappa}-1}{\sqrt{\kappa}+1}$, where $\kappa
=\frac{\beta}{\alpha}$ is the condition number of 
$\widetilde B_m^{\prime}$. 
For the indefinite case, assume that the intervals
$[-\beta_1,-\alpha_1]$ and $[\alpha_2,\beta_2]$ are symmetric with respect to 
the origin, i.e., $\alpha_1=\alpha_2=\alpha$ and $\beta_1=\beta_2=\beta$. 
Then the convergence factor in \eqref{est2} 
is $\frac{\kappa-1}{\kappa+1}$, with an exponent of  $\lbrack\frac{j}{2}\rbrack$. 
These justify the significantly slower convergence behavior of MINRES when applied to indefinite correction equations, as opposed to their definite counterparts.
 
From Theorem~\ref{thm2}, in order to evaluate the convergence of MINRES for  \eqref{correctionm}, it is necessary to determine the spectral intervals $[\alpha,\beta]$, $[-\beta,-\alpha]$ or
$[-\beta_1,-\alpha_1]\cup[\alpha_2,\beta_2]$ for 
$\widetilde B_m^{\prime}$ in \eqref{Mm}. 
Observe that as $\widetilde X_m$ 
approaches $X_m$,  $\widetilde B_m^{\prime}$ converges to
\begin{equation}\label{Bm}
		B_m^{\prime}=Y_m^*(A-\tau I)Y_m=\widehat \Lambda_m-\tau I
\quad\mbox{with}\quad
  \widehat\Lambda_m=\diag\{\lambda_{m+1},\cdots,\lambda_n\},
	\end{equation}  
where the columns of $Y_m=[x_{m+1},\dots,x_n]$ form an orthonormal 
basis for $X_m^{\perp}$. 
Therefore, once $\widetilde X_m$ is sufficiently accurate, the eigenvalues 
$\sigma_1,\dots,\sigma_{n-m}$ of $B_m^{\prime}$, i.e., 
\begin{equation}\label{eigenvaluesm}
	\lambda_{m+1}-\tau,\lambda_{m+2}-\tau,\dots,\lambda_{n}-\tau,  
\end{equation}
reliably approximate those 
$\tilde\sigma_1,\dots,\tilde\sigma_{n-m}$ of $\widetilde B_m^{\prime}$.  
Specifically, denote by $\Phi_m$ 
the canonical angle matrix \cite[p.~329]{golub2012matrix} 
between the range spaces of $\widetilde X_m$ and $X_m$. 
To rigorously quantify how closely $B_m^{\prime}$ approximates 
$\widetilde B_m^{\prime}$, 
we establish the following result.  

\begin{theorem}\label{thm3} 
	The matrix $\widetilde B_m^{\prime}$ is orthogonally similar to $B_m^{\prime}$ 
	up to the error term 
	\begin{equation}\label{similarity2}
		\|\widetilde B_m^{\prime}-W^*B_m^{\prime}W\|\leq 3|\lambda_n-\tau|\|\sin\Phi_m\|^2, 
	\end{equation}
	where $W\in\mathbb{C}^{(m-n)\times(m-n)}$ is a unitary matrix; see \eqref{defW}.  
\end{theorem}

\begin{proof}
Since $[\widetilde X_m,\widetilde Y_m]$ and $[X_m,Y_m]$ are unitary, 
we can orthogonally decompose 
\begin{equation}\label{Ymdecomp}
		\widetilde Y_m = Y_mE+X_mF,
	\end{equation}  
	where $E\in\mathbb{C}^{(n-m)\times(n-m)}$ and 
	$F\in\mathbb{C}^{(n-m)\times m}$ satisfy $E^*E+F^*F=I$. 
By the definition of the sine of the canonical angles between 
	two subspaces, we have 
	\begin{equation}\label{sinpsi}
		\|\sin\Phi_m\|=\|\widetilde Y_m^*X_m\|=\|F\|. 
	\end{equation} 

Substituting 
\eqref{Ymdecomp} into the second relation of \eqref{Mm}, by  \eqref{eigA} and \eqref{Bm}, we obtain
\begin{equation}\label{Bmtilde} 
	\widetilde B_m^{\prime}=(Y_mE+X_mF)^*(A-\tau I)(Y_mE+X_mF)
	=E^*B_m^{\prime}E + F^*(\Lambda_m-\tau I)F,
\end{equation} 
where $\Lambda_m$ is defined in \eqref{block}. 
Let $E=U\Sigma V^*$ be the SVD of $E$, where 
 $U,V\in\mathbb{C}^{(n-m)\times(n-m)}$ are 
unitary, and the diagonal  $\Sigma\in\mathbb{R}^{(n-m)\times(n-m)}$ 
has non-negative diagonal elements. 
It follows from $E^*E+F^*F=I$ that $\|\Sigma\|=\|E\|\leq 1$ and 
\begin{equation}\label{Isigma}
	\|I-\Sigma\|\leq\|I-\Sigma^*\Sigma\|
	=\|I-E^*E\|=\|F^*F\|=\|F\|^2
	=\|\sin\Phi_m\|^2. 
\end{equation} 
Introduce the unitary matrix
\begin{equation}\label{defW}
W=UV^*\in\mathbb{C}^{(n-m)\times(n-m)}.
\end{equation} 
Subtracting $W^*B_m^{\prime}W$ from both sides of \eqref{Bmtilde}, 
taking the $2$-norms on both sides, and applying 
\eqref{sinpsi} and \eqref{Isigma}--\eqref{defW}, we obtain
	\begin{eqnarray*}
		\|\widetilde B_m^{\prime}-W^*B_m^{\prime}W\|
		&=&\|V\Sigma U^*B_m^{\prime}U\Sigma V^*-VU^*B_m^{\prime}UV^*
		+ F^*(\Lambda_m-\tau I)F\| \nonumber\\
		&\leq&\|(\Sigma-I) U^*B_m^{\prime}U\Sigma+U^*B_m^{\prime}U(\Sigma-I)\|
		+ \|F^*(\Lambda_m-\tau I)F\| \nonumber\\
		&\leq&\|\Sigma-I\|\|B_m^{\prime}\| +\|B_m^{\prime}\|\|\Sigma-I\|
		+  \|\Lambda_m-\tau I\|\|F\|^2 \nonumber\\
		&\leq&(2\|B_m^{\prime}\|+\|\Lambda_m-\tau I\|)\|\sin\Phi_m\|^2. %\label{simi}
	\end{eqnarray*} 
Bound \eqref{similarity2} then follows by observing from \eqref{order},  \eqref{block} and \eqref{Bm} that 
$	\|\Lambda_m-\tau I\| 
		\leq|\lambda_n-\tau|$ and $\|B_m^{\prime}\| 
		\leq|\lambda_n-\tau|$.  
\end{proof}

Since $W$ is unitary, 
$W^*B_m^{\prime}W$ shares the same eigenvalues 
$\sigma_1,\dots,\sigma_{n-m}$ of $B_m$ listed in \eqref{eigenvaluesm}. 
Suppose that the eigenvalues $\sigma_i$ and $\tilde\sigma_i$ of $B_m$ 
and $\widetilde B_m$ are arranged in the same (ascending or descending) order. 
Combining Theorem~\ref{thm3} with a classical eigenvalue perturbation theory
\cite[Theorem 10.3.1]{parlett1998symmetric}, we obtain
\begin{equation}\label{eigenbound}
	|\tilde\sigma_i-\sigma_i|
	\leq  \|\widetilde B_m^{\prime}-W^*B_m^{\prime}W\| 
	\leq 3|\lambda_n-\tau|\|\sin\Phi_m\|^2.
\end{equation}

Utilizing Theorems~\ref{thm2} and \ref{thm3}, we are now in a position to 
establish the convergence properties of MINRES for \eqref{correctionm}.
 
\begin{theorem}\label{thm4} 
	Assume that $\widetilde X_m$ is a sufficiently accurate approximation to $X_m$ such that 
	\begin{equation}\label{conditionm}
		\delta_m:=3|\lambda_n-\tau|\|\sin\Phi_m\|^2<|\lambda_{m+1}-\tau|.
	\end{equation}  
(\romannumeral1) If $\tau>\dfrac{\lambda_{\max}+\lambda_{\max,m+1}}{2}$, 
where $\lambda_{\max}$ and $\lambda_{\max,m+1}$ are respectively the $1$st 
and $(m+1)$th largest eigenvalues of $A$, then  
\begin{equation}\label{convbound1}
	\frac{\|r_{\mathrm{in},j}\|}{\|r_1\|} \leq 2\left(
	1-\frac{2}{1+\sqrt{(\tau-\lambda_{\min}+\delta_m)
			/(\tau-\lambda_{\max,m+1}-\delta_m)}}\right)^j;
\end{equation} 
(\romannumeral2) If $\tau<\dfrac{\lambda_{\min}+\lambda_{\min,m+1}}{2}$, 
where $\lambda_{\min}$ and $\lambda_{\min,m+1}$ are respectively the $1$st and $(m+1)$th smallest eigenvalues of $A$, then 
\begin{equation}\label{convbound2}
	\frac{\|r_{\mathrm{in},j}\|}{\|r_1\|} \leq 2\left(
	1-\frac{2}{1+\sqrt{(\lambda_{\max}-\tau+\delta_m)
			/(\lambda_{\min,m+1}-\tau-\delta_m)}}\right)^j;
\end{equation}
(\romannumeral3) If $\dfrac{\lambda_{\min}+\lambda_{\min,m+1}}{2}<
\tau<\dfrac{\lambda_{\max}+\lambda_{\max,m+1}}{2}$, then     
\begin{equation}\label{convbound3}
	\frac{\|r_{\mathrm{in},j}\|}{\|r_1\|} \leq 2\left(
	1-\frac{2}{1+ (|\lambda_{n}-\tau|+\delta_m)
		/(|\lambda_{m+1}-\tau|-\delta_m)}
	\right)^{\lbrack\frac{j}{2}\rbrack},
\end{equation}
where $\lambda_{n}$ is the eigenvalue further from the target $\tau$ 
between $\lambda_{\max}$ and $\lambda_{\min}$.   
\end{theorem}

\begin{proof}
(\romannumeral1) In this case, $\lambda_1,\dots,\lambda_m$ are the $m$ largest eigenvalues $\lambda_{\max}, \dots, \lambda_{\max,m}$ 
of $A$, while all the remaining $n-m$ eigenvalues are smaller than $\tau$, and 
$\lambda_{m+1}=\lambda_{\max,m+1}$, $\lambda_n=\lambda_{\min}$. 
Therefore, all the eigenvalues $\sigma_i$ of $B_m^{\prime}$ presented in 
\eqref{eigenvaluesm} are negative and fall within the interval  
$[\lambda_{\min}-\tau,\lambda_{\max,m+1}-\tau].$  
Relations \eqref{eigenbound}--\eqref{conditionm} imply that the 
perturbed eigenvalues $\tilde\sigma_i$ of $\widetilde B_m^{\prime}$ lie in the negative interval
$$
[-(\tau-\lambda_{\min})-\delta_m,-(\tau-\lambda_{\max,m+1})+\delta_m]. 
$$  
Substituting $\alpha=\tau-\lambda_{\max,m+1}-\delta_m$ 
and $\beta=\tau-\lambda_{\min}+\delta_m$ into \eqref{est1} yields \eqref{convbound1}. 
 
(\romannumeral2) In this case, $\lambda_1,\dots,\lambda_m$ are the 
$m$ smallest eigenvalues $\lambda_{\min},\dots,\lambda_{\min,m}$ of $A$. 
The other $n-m$ eigenvalues are larger than $\tau$, with 
$\lambda_{m+1}=\lambda_{\min,m+1}$ and $\lambda_n=\lambda_{\max}$. 
Then the eigenvalues $\sigma_i$ listed in \eqref{eigenvaluesm} are 
positive and lie in 
$[\lambda_{\min,m+1}-\tau,\lambda_{\max}-\tau].$   
In light of \eqref{eigenbound}--\eqref{conditionm},  
the eigenvalues $\tilde\sigma_i$ of $\widetilde B_m^{\prime}$ reside in the positive interval
$$
[\lambda_{\min,m+1}-\tau-\delta_m,\lambda_{\max}-\tau+\delta_m].
$$ 
Thus, setting $\alpha\!=\!\lambda_{\min,m+1}\!-\!\tau\!-\!\delta_m$ 
and $\beta\!=\!\lambda_{\max}\!-\!\tau\!+\!\delta_m$ in \eqref{est1} 
establishes \eqref{convbound2}. 

(\romannumeral3) Under this assumption, $\lambda_1,\dots,\lambda_m$ are 
interior eigenvalues of $A$, and, by \eqref{ordernew},  $|\lambda_{m+1}-\tau|\leq|\lambda_i-\tau|\leq|
\lambda_n-\tau|$ for $i=m+1,\dots,n$. Consequently, the eigenvalues 
$\sigma_i$ listed in \eqref{eigenvaluesm} distribute within the 
union of disjoint intervals  
$$ 
[-|\lambda_n-\tau|,-|\lambda_{m+1}-\tau|]
\cup[|\lambda_{m+1}-\tau|,|\lambda_n-\tau|]. 
$$
Exploiting \eqref{eigenbound}, the eigenvalues 
$\tilde\sigma_i$ of $\widetilde B_m$ are located within 
$$
[-|\lambda_n-\tau|-\delta_m,-|\lambda_{m+1}-\tau|+\delta_m]
\cup[|\lambda_{m+1}-\tau|-\delta_m,|\lambda_n-\tau|+\delta_m], 
$$
where $|\lambda_{m+1}-\tau|-\delta_m>0$ by \eqref{conditionm}.  
Therefore, substituting $\alpha_1=\alpha_2=|\lambda_{m+1}-\tau|-\delta_m$ and 
$\beta_1=\beta_2=|\lambda_n-\tau|+\delta_m$ into \eqref{est2} delivers \eqref{convbound3}.  
\end{proof}

\begin{remark}\label{remark1}	
By setting $m=1$, \eqref{correctionm} reduces to \eqref{correction}, 
and Theorem~\ref{thm3} provides the convergence 
results of MINRES for the standard correction equation \eqref{correction} when computing $\lambda_{\max},\lambda_{\min}$, and an interior eigenvalue of $A$, where the convergence rates respectively depend on 
	\begin{equation}\label{convrates}\small
		\gamma_1=\frac{\tau-\lambda_{\max,2}-\delta_1}
		{\tau-\lambda_{\min}+\delta_1},\qquad  
		\gamma_2=\frac{\lambda_{\min,2}-\tau-\delta_1}
		{\lambda_{\max}-\tau+\delta_1},\qquad
		\gamma_3=\frac{|\lambda_{2}-\tau|-\delta_1}
		{|\lambda_{n}-\tau|+\delta_1}. 
	\end{equation} 
Assume that $\tilde x_1$ is a sufficiently accurate approximation to the 
desired eigenvector $x_1$, implying that $\delta_1$ defined in 
\eqref{conditionm} is small. 
As the separation between $\tau$ and the adjacent eigenvalue 
$\lambda_{\max,{2}}$, $\lambda_{\min,2}$, or $\lambda_{2}$ increases, 
the corresponding convergence factors 
$1-\frac{2\sqrt{\gamma_1}}{1+\sqrt{\gamma_1}}$, 
$1-\frac{2\sqrt{\gamma_2}}{1+\sqrt{\gamma_2}}$ 
or $1-\frac{2\gamma_3}{1+\gamma_3}$ in 
\eqref{convbound1}--\eqref{convbound3} decreases, thereby 
accelerating MINRES convergence. 
Conversely, MINRES for \eqref{correction} may 
converge very slowly when $\lambda_1$ is poorly separated from those 
adjacent eigenvalues, as noted in section~\ref{sec:2}.    
\end{remark}

\begin{remark}\label{remark2} 
As indicated by \eqref{ordernew}, the distance $|\lambda_{m+1}-\tau|$ 
is considerably large, meaning that condition \eqref{conditionm} 
is fulfilled as soon as $\widetilde X_m$ is sufficiently close to $X_m$. 
In this case, Theorem~\ref{thm3} shows that when computing the 
largest, smallest or interior eigenpairs, the convergence rates 
of MINRES for \eqref{correctionm} respectively depend on 
\begin{equation}\label{convrates2}\small
		\gamma_1^{\prime}=\frac{\tau-\lambda_{\max,m+1}-\delta_m}
		{\tau-\lambda_{\min}+\delta_m},\qquad  
		\gamma_2^{\prime}=\frac{\lambda_{\min,m+1}-\tau-\delta_m}
		{\lambda_{\max}-\tau+\delta_m},\qquad
		\gamma_3^{\prime}=\frac{|\lambda_{m+1}-\tau|-\delta_m}
		{|\lambda_{n}-\tau|+\delta_m}. 
	\end{equation}  
For any sufficiently small $\delta_m\geq\delta_1>0$, these factors are 
obviously larger than their counterparts $\gamma_1$, $\gamma_2$ and 
$\gamma_3$ in \eqref{convrates}.
%, provided that $|\lambda_{m+1}-\tau|\gg 
%|\lambda_{2}-\tau|$. 
%Geometrically, this requires the $m$ clustered eigenvalues $\lambda_1, 
%\dots, \lambda_m$ to be well separated from the rest of the spectrum.  
Therefore, MINRES converges substantially faster for the new 
correction equation \eqref{correctionm} in JD-V than for the 
standard equation \eqref{correction} in JD,  as numerically 
corroborated by the following experiment.	 
\end{remark}

\begin{exper}\label{exper0}
Setting $\tau=0$ and $5$, we seek the smallest and 
an interior eigenpair $(\lambda_1,x_1)$, respectively, of a $10000\times10000$ 
diagonal matrix $A$, whose eigenvalues are uniformly distributed within four clusters of sizes $50$, $4950$, $50$, and $4950$:
%
%partitioned into four clusters. 
%Each cluster consists of uniformly distributed values, 
%with sizes $50$, $4950$, $50$, and $4950$, respectively:   
$$
\underbrace{0.0001,\dots,0.005}_{50}, \qquad
\underbrace{0.05,\dots,4.95}_{4950}, \qquad
\underbrace{5.0001,\dots,5.005}_{50}, \qquad
\underbrace{5.05,\dots,10}_{4950}.
$$  
For both cases, $m=50$, $|\lambda_1-\tau|=0.0001$ and $|\lambda_{m+1}-\tau|=0.05$ in \eqref{ordernew}.  
Suppose that $\sin\angle(\tilde x_1,x_1)=\|\sin\angle(\widetilde 
X_m,X_m)\|=0.01$, meaning that $\delta_1=\delta_m=0.003$ and $0.0015$ for 
$\tau=0$  and $5$, respectively.  
Figure~\ref{fig0} depicts the convergence curves of MINRES for solving  
\eqref{correctionm} and  \eqref{correction} with $\tau=0$ (left) 
and $5$ (right).  
\end{exper}

\begin{figure}[tbhp]
	\centering
	\includegraphics[width=0.49\textwidth]{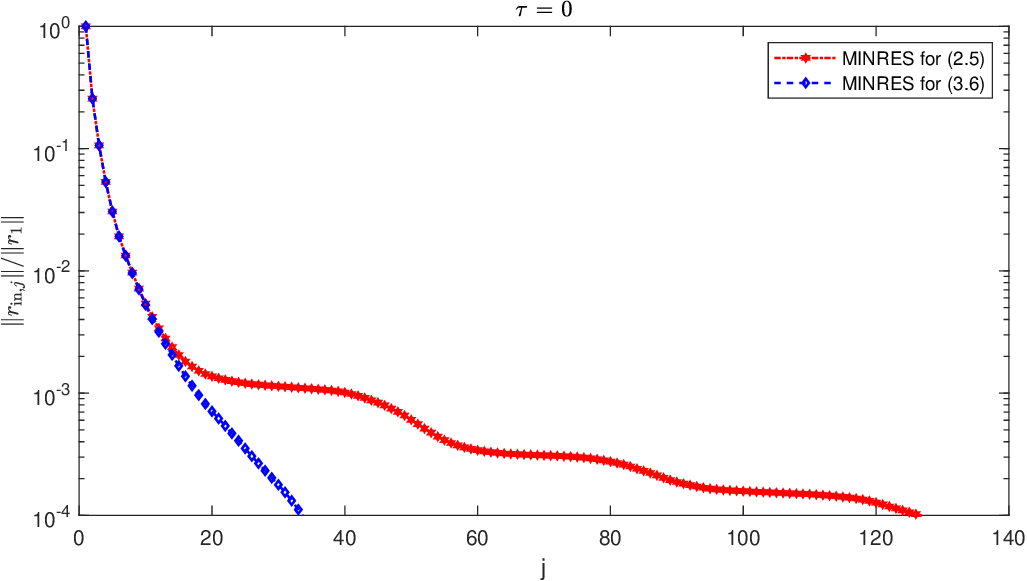}\hfill
	\includegraphics[width=0.49\textwidth]{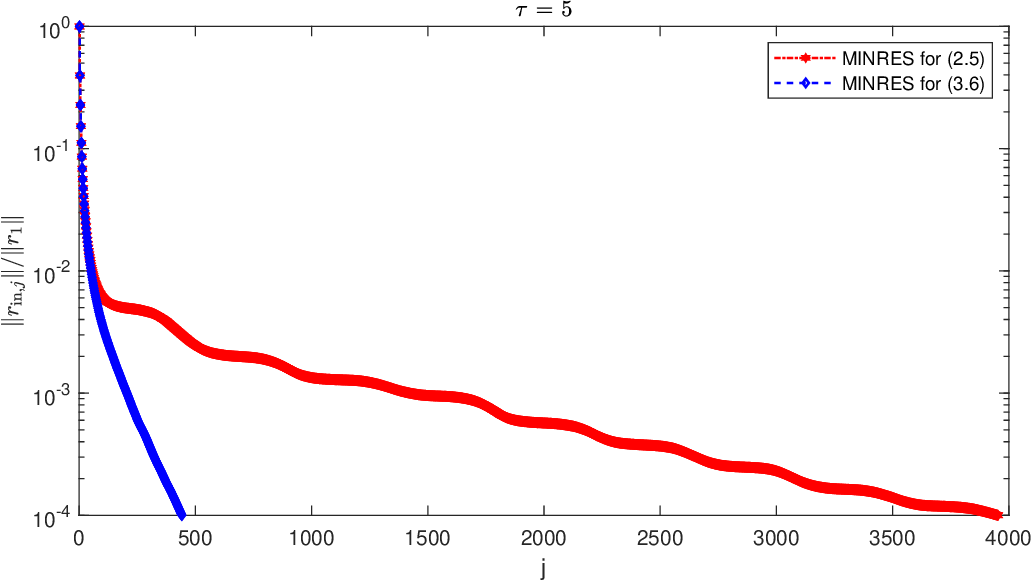} 	  
	\caption{Convergence curves of MINRES for solving \eqref{correction} and \eqref{correctionm}.}\label{fig0}
\end{figure}

Notably, Theorem~\ref{thm4} characterizes the worst-case 
convergence behavior of MINRES for \eqref{correctionm} and \eqref{correction}.  
As illustrated in the figure, while MINRES exhibits
comparable initial convergence for both equations, 
it ultimately converges significantly faster for 
\eqref{correctionm}.   
As a consequence, to achieve a moderate relative residual tolerance of $10^{-4}$, solving \eqref{correctionm} requires only 26.19\% and 11.18\% of the iterations needed for \eqref{correction} when $\tau=0$ and $5$, respectively. 
 
\begin{remark}
When computing the largest, smallest, or an interior eigenvalue of $A$ (corresponding to the cases $\tau\approx\lambda_{\max}$, $\lambda_{\min}$ and $\lambda_{1}$, respectively), assuming $\delta_m$ is sufficiently small, the parameters in \eqref{convrates2} satisfy
	\begin{equation*} \small
		\gamma_1^{\prime}\approx\frac{\lambda_{\max}-\lambda_{\max,m+1}}
		{\lambda_{\max}-\lambda_{\min}},\qquad  
		\gamma_2^{\prime}\approx\frac{\lambda_{\min,m+1}-\lambda_{\min}}
		{\lambda_{\max}-\lambda_{\min}},\qquad
		\gamma_3^{\prime}\approx\frac{|\lambda_{m+1}-\lambda_{1}|}
		{|\lambda_{n}-\lambda_{1}|}.  
	\end{equation*}   
Utilizing $\frac{1}{2}(\lambda_{\max}-\lambda_{\min})\leq |\lambda_{n}-\lambda_{1}|\leq \lambda_{\max}-\lambda_{\min}$, 
we deduce $\frac{\gamma_3^{\prime}}{2}\lessapprox\gamma_1^{\prime}\approx\gamma_2^{\prime}\lessapprox\gamma_3^{\prime}$ 
%\begin{equation*} 
%	\frac{\gamma_3^{\prime}}{2}\lessapprox\gamma_1^{\prime}\approx\gamma_2^{\prime}\lessapprox\gamma_3^{\prime} 
%\end{equation*}
provided that   $\lambda_{\max}-\lambda_{\max,m+1}\approx\lambda_{\min,m+1}-\lambda_{\min}\approx |\lambda_{m+1}-\lambda_{1}|$.    
Given that $\gamma_1^{\prime}, \gamma_2^{\prime},\gamma_3^{\prime}$ are typically smaller than $1$, we have $\gamma_3^{\prime}\lessapprox 2 \gamma_1^{\prime} < \frac{2\sqrt{\gamma_1^{\prime}}}{1+\gamma_1^{\prime}}$ and 
 $\gamma_3^{\prime}\lessapprox 2 \gamma_2^{\prime}  < \frac{2\sqrt{\gamma_2^{\prime}}}{1+\gamma_2^{\prime}}$.  
Therefore, the theoretical convergence factors in  \eqref{convbound1}--\eqref{convbound3} satisfy 
$$1-\frac{2\sqrt{\gamma_1^{\prime}}}{1+\sqrt{\gamma_1^{\prime}}}\approx
1-\frac{2\sqrt{\gamma_2^{\prime}}}{1+\sqrt{\gamma_2^{\prime}}} < \left(1-\frac{2\gamma_3^{\prime}}{1+\gamma_3^{\prime}}\right)^{\frac{1}{2}}.$$ 
This indicates that MINRES for \eqref{correctionm} generally converges significantly faster for extreme eigenpairs than for interior ones, as fully corroborated by Figure~\ref{fig0}.  
\end{remark}
 
\subsection{Construction of new correction equations}\label{subsec:3}
The preceding analysis highlights the advantages of employing 
MINRES to solve the proposed correction equation \eqref{correctionm} 
over the standard \eqref{correction}, especially when $A$ has a cluster of eigenvalues near $\tau$ and sufficiently accurate approximations to the associated eigenvectors are available during the extraction phase.   
From a computational standpoint, it remains to establish rigorous 
numerical criteria for determining the cluster size $m$,  adaptively 
construct \eqref{correctionm}, and develop a practical JD-V algorithm. 

Given a reasonably small tolerance $\epsilon_1>0$, an eigenvalue 
$\lambda_i$ of $A$ is classified as being clustered near the target $\tau$ if  
\begin{equation}\label{pretol1}
	|\lambda_i-\tau|\leq\max\{|\tau|,1\}\cdot\epsilon_1.   
\end{equation}
By \eqref{ordernew}, $m$ is uniquely defined with  $|\lambda_{m}-\tau| \leq \max\{|\tau|,1\} \cdot 
\epsilon_1 <|\lambda_{m+1}-\tau|$. 
At the $k$th outer iteration of JD-V, we identify a Ritz 
value $\theta_i$ ($i=2,\dots,k$) obtained from 
\eqref{projection}--\eqref{Ritzpair} as an approximation to a clustered 
eigenvalue if
\begin{equation}\label{select1}
	|\theta_i-\tau|\leq\max\{|\tau|,1\}\cdot\epsilon_1. 
\end{equation}

Recall from section~\ref{subsec:1} that for \eqref{correctionm} and 
\eqref{correction} to yield equally effective subspace expansions, 
the higher order term $R_{2:m}\widetilde X_{2:m}^*t$ omitted from 
\eqref{derive1} must be negligible compared to $-r_1$. 
In light of \eqref{R_3}, this requires the block residual $R_{2:m}$  
to be sufficiently small.  
Furthermore, Theorem~\ref{thm4} dictates that for the theoretical bounds  
\eqref{convbound1}--\eqref{convbound3} to be meaningful, and for MINRES to converge significantly faster for \eqref{correctionm} than for
\eqref{correction}, the parameter $\delta_m$ defined in \eqref{conditionm} 
must be adequately small. 
Together, these conditions imply that the approximate eigenvectors $\tilde x_2,\dots,\tilde x_m$ participating in \eqref{correctionm} 
must have at least moderate accuracy. 
Therefore, if the residual $r_i$ defined in \eqref{extraction} of a Ritz pair $(\theta_i,\tilde x_i)$ satisfies 
\begin{equation}\label{select2} 
	\|r_i\|\leq\|A\|_1 \cdot\epsilon_2, 
\end{equation}
for a moderately small threshold $\epsilon_2$,
we accept $(\theta_i,\tilde x_i)$ as a reliable approximation to an eigenpair of $A$. In practice, it suffices to employ 
$\epsilon_2\gg \varepsilon_{\mathrm{out}}$, 
the outer stopping tolerance 
$\varepsilon_{\mathrm{out}}$ in \eqref{converg}. 
For example, when $\varepsilon_{\mathrm{out}}\in[10^{-14},10^{-8}]$,  numerical experiments indicate that $\epsilon_2\in[10^{-3},10^{-2}]$ serves as a robust practical choice. 

Assuming that there are $\widetilde m-1$ indices $i_2,\dots,i_{\widetilde m}$ from $2,\dots,k$ fulfilling both \eqref{select1} and \eqref{select2}, 
we construct
\begin{equation}\label{tildeX}
	\widetilde X_{\widetilde m}=
	[\tilde x_1,\tilde x_{i_2},\dots,\tilde x_{i_{\widetilde m}}], 
\end{equation}    
and  replace $\widetilde X_m$ with $\widetilde X_{\widetilde m}$ when 
forming  \eqref{correctionm}.

\begin{remark}\label{remark9}
For a fixed $\epsilon_2\gg\varepsilon_{\mathrm{out}}$, the selected 
approximate eigenvectors $\tilde x_{i_2},\dots$, 
$\tilde x_{i_{\widetilde m}}$ may have poorer accuracy than 
$\tilde x_1$ as the outer iterations proceed, particularly when 
$\varepsilon_{\mathrm{out}}<\frac{\|r_1\|}{\|A\|_1} 
\ll \frac{\|r_i\|}{\|A\|_1}\leq\epsilon_2$	
for some $i\in\{i_2,\dots,i_{\widetilde m}\}$. 
In this case, $\|R_{i_2:i_{\widetilde m}}\|$ can significantly exceed 
$\|r_1\|$. Consequently, omitting $R_{i_2:i_{\widetilde m}}\widetilde 
X_{i_2:i_{\widetilde m}}^*t$ from \eqref{derive1} renders the solution to 
the resulting correction equation \eqref{correctionm} marginally less 
effective in expanding subspace than that of \eqref{correction}. 
As a result, JD-V may converge slightly slower than JD, as demonstrated in Figure~\ref{fig0b}. 
As a part of Experiment~\ref{exper1}, the figure plots the outer convergence histories of both methods for computing the largest (left), smallest (middle), and an interior (right) eigenpair of the ``barth'' matrix  from the SuiteSparse 
Matrix Collection \cite{davis2011university}, using the shifts  $\tau=7.4$, $-2.3$, and $0$, respectively.   
Ultimately, such a marginal delay in outer convergence is acceptable, given that the overall efficiency is substantially improved. 
\end{remark}

\begin{figure}[tbhp]
	\centering
	\includegraphics[width=0.32\textwidth]{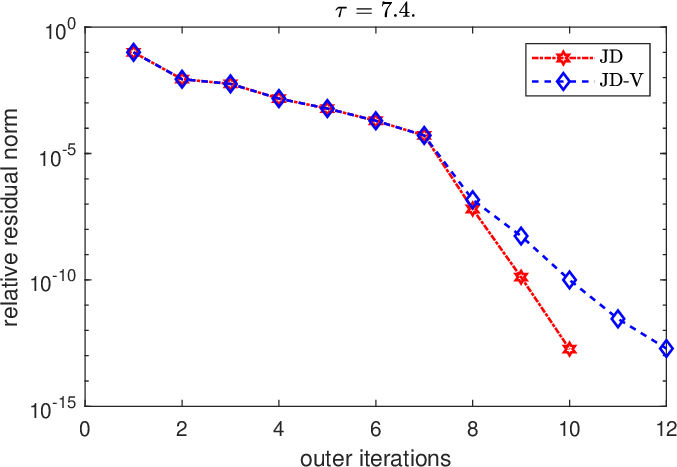}\hfill 
	\includegraphics[width=0.32\textwidth]{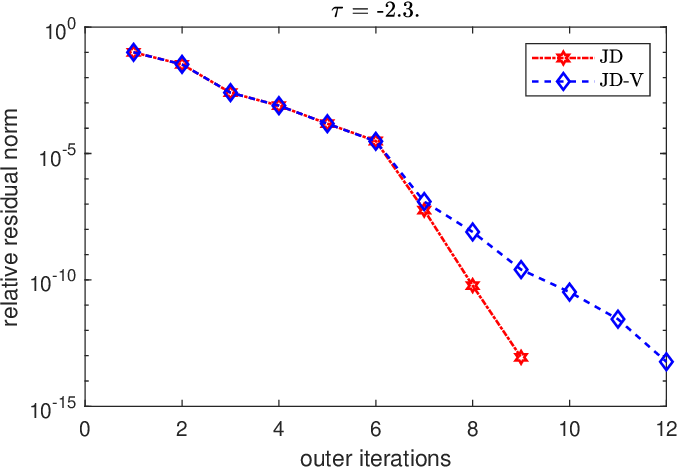}\hfill 
	\includegraphics[width=0.32\textwidth]{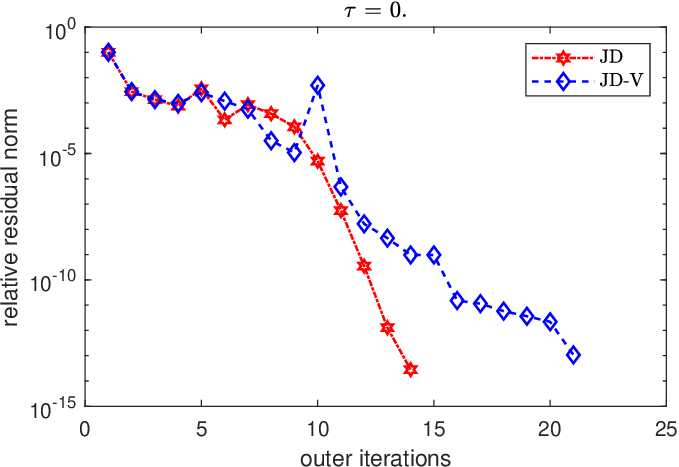} 
	\caption{Computing one eigenpair of ``barth'' using JD-V and JD.}\label{fig0b}
\end{figure}

\begin{remark}
Given $\epsilon_2>\varepsilon_{\mathrm{out}}$, the choice of $\epsilon_1>0$ is critical. 
While a larger $\epsilon_1$ expands the cluster size 
$\widetilde m$ in \eqref{tildeX} and 
accelerates MINRES convergence for \eqref{correctionm}, 
it simultaneously increases the cost of matrix-vector 
multiplications involving the coefficient matrix $\widetilde B_{\widetilde m}$ per MINRES iteration, particularly for sparse $A$.
To balance this trade-off, we empirically recommend choosing   $\epsilon_1\in[0.01, 0.05]$. 
\end{remark}

\begin{remark}\label{remark3}
Two extreme cases are $\epsilon_1=\epsilon_2=0$ and 
$\epsilon_1=\epsilon_2=+\infty$. 
In the first case, $\widetilde{m}\equiv m\equiv1$, and JD-V reduces to the 
standard JD method. In the second case, $\widetilde{m}=k$, meaning all 
Ritz vectors are incorporated into \eqref{correctionm}. 
The latter approach, equivalent to the one in \cite{genseberger1999alternative} and denoted as JD-V($+\infty$)  hereafter, is known to be computationally inefficient \cite{deSturler2002ImprovingTC,genseberger1999alternative}.  
This inefficiency arises because the \emph{entire} subspace $\spans\{\widetilde X_k\}$ is typically a poor approximation to any eigenspace of $A$. Hence, $R_{2:k}\widetilde X_{2:k}^*t$ dominates $-r_1$ in \eqref{derive1}, causing the solution of the relevant correction equation \eqref{correctionm} to deviate significantly from the standard correction \eqref{correction}.  
As a result, JD-V($+\infty$) expands the subspace ineffectively 
and demands significantly more outer iterations to 
converge than the standard JD and JD-V, a phenomenon confirmed numerically in Experiments~\ref{exper1}--\ref{exper2}.
\end{remark}

For $k$ small, there maybe no Ritz pair satisfy \eqref{select1}--\eqref{select2}. In this case, $\widetilde 
m=1$, and \eqref{correctionm} reduces to the standard \eqref{correction}. 
Nevertheless, for $m>1$, as $k$ grows, more Ritz pairs begin to satisfy 
these criteria. Consequently, incorporating multiple Ritz vectors 
($\widetilde m>1$) into \eqref{correctionm} enables JD-V to outperform JD in overall efficiency. Ultimately, $\widetilde m$ reaches $m$ 
as JD-V approaches convergence. 

\section{A thick-restart JD-V algorithm with deflation and purgation}\label{sec:4}
To compute $\ell>1$ eigenpairs of $A$ near $\tau$, 
we propose an efficient 
thick-restart JD-V algorithm incorporating adaptive shifts, deflation, and purgation. 

\subsection{Thick restart}\label{subsec:4}
For practical purpose, we adapt the thick-restart strategy introduced in  
\cite{huang2026preconditioning} for JDSVD-V method to the JD-V algorithm.  

Specifically, given the user-prescribed maximum and minimum subspace dimensions $k_{\max}$ and $k_{\min}$,  
when $k=k_{\max}$, we retain a restarting subspace $\XX_{\new}\subset \XX$ of dimension 
$$
k_{\new}=\max\{k_{\min},\widetilde m\},
$$
where $\widetilde m$ is as in \eqref{tildeX}.  
The orthonormal basis matrix $\widetilde X$ of $\XX_{\new}$ and the associated projected matrix $H$ are updated as follows: 
\begin{itemize}
	\item[(i)] If $\widetilde m<k_{\min}$, we set $\widetilde X=[\tilde x_1,\dots,\tilde x_{k_{\min}}]$ 
	and $H=\diag\{\theta_1,\dots,\theta_{k_{\min}}\}$; 
	\item[(ii)] If $\widetilde m\geq k_{\min}$, then $\widetilde X=[\tilde x_1,\tilde x_{i_2},\dots,\tilde x_{i_{\widetilde m}}]$ 
	and $H=\diag\{\theta_1,\theta_{i_2},\dots,\theta_{i_{\widetilde m}}\}$. 
\end{itemize} 

Initially, when $\widetilde m<k_{\min}$, the restarting subspace $\XX_{\new}$ is spanned by the $k_{\min}$ Ritz vectors associated with the Ritz values closest to $\tau$. As the iterations proceed and $\widetilde m\geq k_{\min}$, $\XX_{\new}$ adaptively retains all sufficiently accurate Ritz vectors within the target cluster.  
Consequently, for $k_{\min}<m\ll k_{\max}$, the thick-restart JD-V 
algorithm requires substantially fewer outer iterations than the  
conventional scheme with a fixed restart dimension $k_{\new}\equiv k_{\min}$. 
 
\subsection{Deflation}\label{subsec:5} 
The JD-V method naturally incorporates deflation when computing several  eigenpairs $(\lambda_i, x_i)$, $i=1, \dots, \ell$ closest to the target $\tau$. 
 
Assume that $\ell_{\mathrm{c}}<\ell$ approximate eigenpairs 
$(\theta_{1,\mathrm{c}},\tilde x_{1,\mathrm{c}}),\dots,
(\theta_{\ell_{\mathrm{c}},\mathrm{c}}, \tilde x_{\ell_{\mathrm{c}}, \mathrm{c}})$ 
  have converged to the desired $(\lambda_1,x_1)$, $\dots$, 
$(\lambda_{\ell_{\mathrm{c}}},x_{\ell_{\mathrm{c}}})$,  
with the approximate eigenvectors mutually orthonormal. 
Here,  the subscript ``$\mathrm{c}$'' indicates converged quantities. Then
\begin{equation}\label{convergRitz}
	(\Lambda_{\mathrm{c}},X_{\mathrm{c}})=
	(\diag\{\theta_{1,\mathrm{c}},\dots,\theta_{\ell_{\mathrm{c}},\mathrm{c}}\},
	[\tilde x_{1,\mathrm{c}},\dots,\tilde x_{\ell_{\mathrm{c}},\mathrm{c}}])
\end{equation} 
is a converged approximation to the block eigenpair 
$(\Lambda_{\ell_{\mathrm{c}}}, X_{\ell_{\mathrm{c}}})$ of $A$; see  \eqref{block}. 
By definition, the next desired  $(\lambda_{\ell_{\mathrm{c}}+1},x_{\ell_{\mathrm{c}}+1})$ 
is an eigenpair of  $(I-X_{\ell_{\mathrm{c}}}X_{\ell_{\mathrm{c}}}^*)
A(I-X_{\ell_{\mathrm{c}}}X_{\ell_{\mathrm{c}}}^*)$ restricted 
to $X_{\ell_{\mathrm{c}}}^{\perp}$ with the eigenvalue closest to $\tau$.
By a continuity argument, we can apply JD-V to $\widetilde A = (I-X_{\mathrm{c}}X_{\mathrm{c}}^*)A(I-X_{\mathrm{c}}X_{\mathrm{c}}^*)$ 
restricted to $X_{\mathrm{c}}^{\perp}$ to compute it. 
 
By enforcing the search subspace $\XX\perp \spans\{X_{\mathrm{c}}\}$, the 
extraction phase of the deflated JD-V method remains identical to that in section~\ref{sec:2}. 
We take the Ritz pair $(\theta_1,\tilde x_1)$ from \eqref{Ritzpair} 
as an approximate to the desired 
$(\lambda_{\ell_{\mathrm{c}}+1},x_{\ell_{\mathrm{c}}+1})$. 
The corresponding residuals with respect to the eigenproblems of $A$ and 
$\widetilde A|_{X_{\mathrm{c}}^{\perp}}$ are $r_1$ in \eqref{extraction} 
and $\tilde r_1=(I-X_{\mathrm{c}}X_{\mathrm{c}}^*)r_1$, respectively. 
If $(\theta_1,\tilde x_1)$ has not yet converged according to \eqref{converg}, we determine the $\widetilde m-1$ indices $i_2,\dots,i_{\widetilde m}$ based on 
\eqref{select1}--\eqref{select2}, and construct $\widetilde X_{\widetilde m}$ via \eqref{tildeX}. 
We then employ MINRES to solve the modified correction equation \eqref{correctionm}, 
replacing $A$ and $r_1$ with $\widetilde A$ and  $\tilde r_1$, 
respectively. 
Since $\tilde x_i\in\XX\perp\spans\{X_{\mathrm{c}}\}$ for all $i=1,\dots,k$,  
the deflated JD-V correction equation takes the form 
\begin{equation}\label{correctionc}
	(I-\widetilde X_{\mathrm{d}}\widetilde X_{\mathrm{d}}^*)
	(A-\tau I)(I-\widetilde X_{\mathrm{d}}\widetilde X_{\mathrm{d}}^*) t = -\tilde r_1
	\quad\mbox{with}\quad
	t\perp \widetilde X_{\mathrm{d}}=[X_{\mathrm{c}},\widetilde X_{\widetilde m}],
\end{equation}    
where the subscript ``$\mathrm{d}$'' denotes deflation. 
We terminate the inner  MINRES iterations for solving \eqref{correctionc} 
when the inner residual $r_{\mathrm{in}}$ satisfies
\begin{equation}\label{innstop}
	\|r_{\mathrm{in}}\|\leq \|\tilde r_{1}\| \cdot \min\{\beta \varepsilon_{\mathrm{in}},0.1\}, 
\end{equation}
where $\beta$ depends on the current Ritz values, and the threshold 
$\varepsilon_{\mathrm{in}}\in[10^{-4},10^{-3}]$; see  
\cite{jia2014inner,jia2015harmonic}. 
We then use the approximate solution $\tilde t$ to expand the search subspace: $\XX_{\new}=\spans\{\XX,\tilde t\}$, which is inherently orthogonal to 
$\spans\{X_{\mathrm{c}}\}$ since both $\tilde t$ and $\XX$ are so. 
New Ritz pairs are then extracted from $\XX_{\new}$ until $(\theta_1,\tilde x_1)$ converges. 

Upon convergence,  we assign
$(\theta_{\ellc+1,\mathrm{c}},\tilde x_{\ellc+1,\mathrm{c}})
=(\theta_1,\tilde x_1)$, 
append it to the converged block eigenpair $(\Lambda_{\mathrm{c}},X_{\mathrm{c}})
=(\diag\{\Lambda_{\mathrm{c}},\theta_{\ellc+1,\mathrm{c}} \},
[X_{\mathrm{c}},\tilde x_{\ellc+1,\mathrm{c}}])$, and increment $\ellc$ by $1$. 
Repeat until all $\ell$ desired eigenpairs of $A$ are computed. 
 
\subsection{Purgation}\label{subsec:8}
When computing $(\lambda_{\ellc+1}, 
x_{\ellc+1})$ for $\ellc\geq 1$, the search subspace $\XX$ 
from the preceding convergence cycle for $(\lambda_{\ellc},x_{\ellc})$ 
typically retains rich information about $x_{\ellc+1}$. 
Rather than initializing a new subspace from scratch, we purge the 
recently converged Ritz vector $\tilde x_{1}=\tilde 
x_{\ellc,\mathrm{c}}$ from $\XX$, and employ the reduced subspace 
$\XX_{\new}=\spans\{\tilde x_2,\dots,\tilde x_{k}\}$ as the initial 
search subspace. 
Clearly, this $\XX_{\new}$ is inherently orthogonal to $\spans\{X_c\}$.    
An orthonormal basis for $\XX_{\new}$ and the corresponding projected 
matrix $H$ are naturally given by
\begin{equation*}
	\widetilde X=[\tilde x_2,\dots,\tilde x_{k}]
	\qquad\mbox{and}\qquad
	H= \diag\{\theta_2,\dots,\theta_k\}.
\end{equation*} 
Referred to as purgation \cite{huang2019inner}, this technique leverages  
the reduced subspace to substantially reduce the total number of outer iterations required for computing the $2$nd through $\ell$th eigenpairs.  
 
\subsection{Adaptively changing shift}\label{subsec:6}
As $\ellc$ increases, the deflated operator $\widetilde A|_{\widetilde X_{\mathrm{c}}^{\perp}}$ has $m-\ellc$ clustered eigenvalues near $\tau$ (i.e., $\lambda_{\ellc+1},\dots,\lambda_{m}$), dictating 
$1\leq \widetilde m\leq m-\ellc$ for $\ellc<m$. 
When $\ellc\geq m$, we have $\widetilde m\equiv1$, which reduces   
\eqref{correctionc} to the standard deflated correction equation, 
and simplifies JD-V to JD. 
While this enables rapid MINRES convergence, the outer iterations 
often converge slowly.  
To accelerate outer convergence and enhance overall efficiency, we 
substitute the target $\tau$ in \eqref{correctionc} with a dynamic 
shift $\theta_1$ once it is sufficiently accurate. Specifically, when  
\begin{equation}\label{switch}
	\|r_1\|\leq \|A\|_1 \cdot \varepsilon_{\mathrm{\tau}} 	
\end{equation} 
for a user-prescribed tolerance $\varepsilon_{\mathrm{\tau}} >\varepsilon_{\mathrm{out}}$, we form $\widetilde X_{\widetilde m}=[x_1,x_{i_2},\dots,x_{i_{\widetilde m}}]$ with the indices $i_2,\dots,i_{\widetilde m}$ determined via \eqref{select1}--\eqref{select2} by replacing $\tau$ with $\theta_1$. 
The deflated correction equation with this adaptive shift then becomes
\begin{equation}\label{correctiont}
	(I-\widetilde X_{\mathrm{d}}\widetilde X_{\mathrm{d}}^*)
	(A-\theta_1 I)(I-\widetilde X_{\mathrm{d}}\widetilde X_{\mathrm{d}}^*) t = -\tilde r_{1}
	\quad\mbox{with}\quad
	t\perp \widetilde X_{\mathrm{d}}.
\end{equation} 

The convergence theory established in section~\ref{subsec:2} applies to 
\eqref{correctiont} by replacing $\tau$ with $\theta_1$, 
implying that MINRES converges substantially faster for \eqref{correctiont} 
than for the fixed-shift equation \eqref{correctionc}, provided that 
the desired eigenvalue $\lambda_{\ell_{\mathrm{c}}+1}$ is well separated from   $\lambda_1,\dots,\lambda_{\ell_{\mathrm{c}}}$, yet tightly  
clustered with the remaining eigenvalues of $A$.

\subsection{Thick-restart JD-V with deflation and purgation, a pseudocode}\label{subsec:7}
Algorithm~\ref{alg1} sketches the thick-restart JD-V algorithm  
with deflation and purgation. 

\begin{algorithm}[htbp]
	\caption{Thick-restart JD-V with deflation and purgation for the target $\tau$.}
	\begin{algorithmic}[1]\label{alg1}
		\STATE{Initialization: Set $X_{\mathrm{c}}=[\ ]$, $\Lambda_{\mathrm{c}}=[\ ]$, $\ellc=0$, $\widetilde X=[\ ]$, $H=[\ ]$, $k=0$, and $x_+=x_0$.}
		
		\WHILE{$\ellc<\ell$}
		
		\STATE{Update $\widetilde X=[\widetilde X,x_+]$ and $H=\widetilde X^*A\widetilde X$, and set $k=k+1$.}
		
		\STATE{\label{step}Compute the eigenvalue decomposition \eqref{projection} of $H$. Form  the Ritz pair  
			$(\theta_1,\tilde x_1)=(\theta_1,\widetilde Xs_1)$ and the 
			residual $r_1=A\tilde x_1-\theta_1\tilde x_1$.}
		
		\IF{$\|r_1\|\leq\|A\|_1\cdot \varepsilon_{\mathrm{out}}$}
		
		\STATE{Update $\Lambda_{\mathrm{c}}
			=\diag\{\Lambda_{\mathrm{c}},\theta_1\}$ and
			$X_{\mathrm{c}}=[X_{\mathrm{c}},\tilde x_1]$, and set $\ellc=\ellc+1$.}
		
		\STATE{\textbf{if} $\ellc=\ell$, \textbf{then} return 
			$(\Lambda_{\mathrm{c}},X_{\mathrm{c}})$ and terminate; \textbf{else} 
			update $\widetilde X=[\tilde x_2,\dots,\tilde x_k]$ and $H=\diag\{\theta_2,\dots,\theta_{k}\}$, set 
			$k=k-1$,  and go to step~\ref{step}. \textbf{fi}\label{step2}}
		
		\ENDIF
		
		\STATE{\textbf{if} $\|r_1\|\leq\|A\|_1\cdot\varepsilon_{\tau}$, \textbf{then} 
			set the inner shift $\rho=\theta_1$; \textbf{else} set $\rho=\tau$. \textbf{fi}}
		
		\STATE{\label{step3}Determine the $\widetilde m -1$ indices $i_2,\dots,i_{\widetilde m}\in\{2,\dots,k\}$ according to \eqref{select1}--\eqref{select2} with $\tau$ replaced by $\rho$, 
			and form  $\widetilde X_{\widetilde m}=
			[\tilde x_1,\tilde x_{i_2},\dots,\tilde x_{i_{\widetilde m}}]$.}
		
		\STATE{\label{step4}
			Set $\widetilde X_{\mathrm{d}}=[X_{\mathrm{c}},\widetilde X_{\widetilde m}]$ and $\tilde r_{1}=(I-X_{\mathrm{c}}X_{\mathrm{c}}^*)r_1$.  
			Use MINRES to solve 
			\begin{equation}\label{correctionfinal} 
				(I-\widetilde X_{\mathrm{d}}\widetilde X_{\mathrm{d}}^*)
				(A-\rho I)(I-\widetilde X_{\mathrm{d}}\widetilde X_{\mathrm{d}}^*) t
				=-\tilde r_{1}
				\qquad\mbox{for}\qquad
				t\perp \widetilde X_{\mathrm{d}} 
			\end{equation} 
			until the inner residual $r_{\mathrm{in}}$ of 
			the approximate solution $\tilde t$ satisfies \eqref{innstop}.} 
		
		\STATE{\textbf{if} $k=k_{\max}$ \textbf{then} perform the 
			thick restart and reset $k=\max\{k_{\min},\widetilde m\}$. 
			\textbf{fi}.}
		
		\STATE{Orthonormalize $\tilde t$ against $\widetilde X$ to 
			obtain 
			$x_+=\frac{(I-\widetilde X\widetilde X^*)\tilde t}
			{\|(I-\widetilde X\widetilde X^*)\tilde t\|}$.} 
		
		\ENDWHILE  
	\end{algorithmic}
\end{algorithm}

To reduce computational costs during the extraction phase, we do not 
explicitly evaluate all Ritz pairs and their residuals. 
Rather, we selectively compute those associated with the principal Ritz 
value $\theta_1$ at Step~\ref{step}, and those with Ritz values nearest to 
the inner shift $\rho=\tau$ or $\theta_1$ at Step~\ref{step3}.
As the eigenpairs of the diagonal matrix $H$ are trivially known when 
turning from step~\ref{step2} to step~\ref{step}, it is sufficient to 
compute the residual $r_1$ for the target pair $(\theta_1,\tilde x_1)$.   
In step~\ref{step4}, solving \eqref{correctionfinal} via MINRES involves
matrix-vector multiplications between the coefficient matrix 
$\widetilde B_{\widetilde m+\ellc}$ and vectors 
$\bar t\in\mathcal{K}_j(B_{\widetilde m+\ellc},\tilde r_1)$ for $j\geq1$. 
Clearly, $\widetilde B_{\widetilde m+\ellc}\bar t=\widehat B_{\widetilde m+\ellc}\bar t$ with   
\begin{equation*}
 	\widehat B_{\widetilde m+\ellc} = 
 	(I-\widetilde X_{\mathrm{d}}\widetilde X_{\mathrm{d}}^*)
 	(A-\rho I). 	 
\end{equation*}   
Therefore, we equivalently apply the one-sided projected matrix $\widehat B_{\widetilde m+\ellc}$ at each MINRES step, which is substantially more efficient than applying the two-sided projection $\widetilde B_{\widetilde m+\ellc}$, especially when 
$A$ is highly sparse and the dimension $\widetilde m+\ellc$ is large. 

%it is equivalent to apply the one-sided projected matrix $\widehat B_{\widetilde m+\ellc}$ at each MINRES step. This is substantially more efficient than utilizing the two-sided projected matrix $\widetilde B_{\widetilde m+\ellc}$, especially when 
%$A$ is highly sparse and the dimension $\widetilde m+\ellc$ is relatively large.  
  
Given a target $\tau$, Algorithm~\ref{alg1} requires as inputs a routine 
for computing matrix-vector products with $A$, the number $\ell$ of 
desired eigenpairs, and the outer stopping tolerance $\varepsilon_{\mathrm{out}}$ 
in \eqref{converg}. Upon convergence, it outputs an approximate 
block eigenpair $(\Lambda_{\mathrm{c}},X_{\mathrm{c}})$ of size $\ell$,  
with $X_{\mathrm{c}}$ having orthonormal columns, satisfying 
\begin{equation*} 
	\|AX_{\mathrm{c}}-X_{\mathrm{c}}\Lambda_{\mathrm{c}}\|_\mathrm{F}
	\leq\sqrt{\ell}\|A\|_1\cdot\varepsilon_{\mathrm{out}}, 
\end{equation*}
where $\|\cdot\|_\mathrm{F}$ denotes the Frobenius norm. Optional 
parameters include the normalized starting vector $x_0\in\mathbb{C}^n$, 
the subspace dimension bounds $k_{\max}$ and $k_{\min}$, the inner 
stopping tolerance $\varepsilon_{\mathrm{in}}$ in \eqref{innstop}, 
the threshold $\varepsilon_{\tau}$ in \eqref{switch} governing the  transition from the fixed target $\tau$ to the adaptive shift $\theta_1$,  
and the clustering and accuracy tolerances $\epsilon_1$ and $\epsilon_2$ 
in \eqref{select1}--\eqref{select2}.  
The default values are set as $x_0=\frac{1}{\sqrt{n}}[1,\dots,1]^*$, $k_{\max}=30$, 
$k_{\min}=3$, $\varepsilon_{\mathrm{in}}=10^{-3}$, $\varepsilon_{\tau}=10^{-4}$ and $\epsilon_1=0.05$, 
$\epsilon_2=0.01$. 
Additionally, MINRES utilizes the zero vector as its initial guess when 
solving the correction equation \eqref{correctionfinal}. 
 
\section{Numerical examples}\label{sec:5}
We report numerical experiments on several test problems to illustrate 
the effectiveness and efficiency of the thick-restarted JD-V 
algorithm for computing several eigenpairs of a large Hermitian matrix $A$ 
near $\tau$. The algorithm was developed in MATLAB. 
All the experiments were implemented on an Intel (R) core (TM) 
i9-10885H CPU 2.40 GHz with the main memory 64 GB and 16 cores 
using the MATLAB R2021a with the machine precision 
$\epsilon=2.22\times 10^{-16}$ under the Microsoft Windows 10 
64-bit system.

%\begin{table}[tbhp]
%	\caption{Properties of the test matrices from the
%		SuiteSparse Matrix Collection \cite{davis2011university}: part I.}\label{table00}
%	\begin{center}
%		\begin{tabular}{ccccc} \toprule
%			$A$&$n$&$\textit{nnz}$&$\lambda_{\max}$&$\lambda_{\min}$  \\ \midrule
%			e40r0100 &6691&46187&7.32&-2.27  \\
%			 
%			\bottomrule
%		\end{tabular}
%	\end{center}
%\end{table}
%
%Table~\ref{table0} lists the test matrices from the SuiteSparse Matrix Collection \cite{davis2011university} 
%together with some of their basic properties, where $nnz$ denotes the total number of
% nonzero entries in $A$, and the largest and smallest eigenvalues $\lambda_{\max}$ and $\lambda_{\min}$ of
% $A$ are, only for experimental purpose, computed by the MATLAB functions {\sf eigs} and
%{\sf eig} for the very large matrices ``??'', ``??'', and all the other moderately large ones, respectively.
\begin{exper}\label{exper1}
	Compute the $20$ largest, smallest and interior clustered eigenpairs of the $6691\times 6691$ sparse ``barth'' matrix  
	(with $46187$ nonzero entries) from the SuiteSparse Matrix Collection 
	\cite{davis2011university} using the {\em unrestarted} standard JD 
	($\epsilon_1=\epsilon_2=0$), JD-V ($\epsilon_1=0.01$ and 
	$\epsilon_2=0.05$) and JD-V($+\infty$) ($\epsilon_1=\epsilon_2=+\infty$) 
	algorithms, with the targets $\tau=7.4$, $-2.3$, and $0$, respectively. 
	Set the outer stopping tolerance to $\varepsilon_{\mathrm{out}}=10^{-12}$, 
	and keep all other parameters at their default values.   
\end{exper}

\begin{table}[tbhp]
	\caption{Results of the unrestarted JD-type algorithms on ``barth'' with $\ell=20$.}\label{table1a} 
	\begin{center}
		\begin{tabular}{cccccccccc} \toprule
			\multirow{2}{*}{$\tau$}
			&\multicolumn{3}{c}{JD}
			&\multicolumn{3}{c}{JD-V} 
			&\multicolumn{3}{c}{JD-V($+\infty$)} 	\\
			\cmidrule(lr){2-4} \cmidrule(lr){5-7}\cmidrule(lr){8-10}
			&$i_{\mathrm{out}}$ &$i_{\mathrm{in}}$ &time  
			&$i_{\mathrm{out}}$ &$i_{\mathrm{in}}$ &time  
			&$i_{\mathrm{out}}$ &$i_{\mathrm{in}}$ &time   \\   \midrule 
			\ 7.4  &\textbf{105}  &4584  &\textbf{0.73}  &129  &\textbf{3502}  &1.20  &363 &5263  &9.76 \\
			-2.3  &\textbf{105}  &5653  &\textbf{0.82}  &126  &\textbf{4007}  &1.09  &490  &8877  &18.9 \\
			0 	 &\textbf{121}  &520280  &44.0  &218  &\textbf{306181}  &\textbf{32.3}  &251  &316460  &36.2 \\
			\bottomrule
		\end{tabular}
	\end{center}
\end{table}

\begin{figure}[tbhp]
	\centering
	\includegraphics[width=0.48\textwidth]{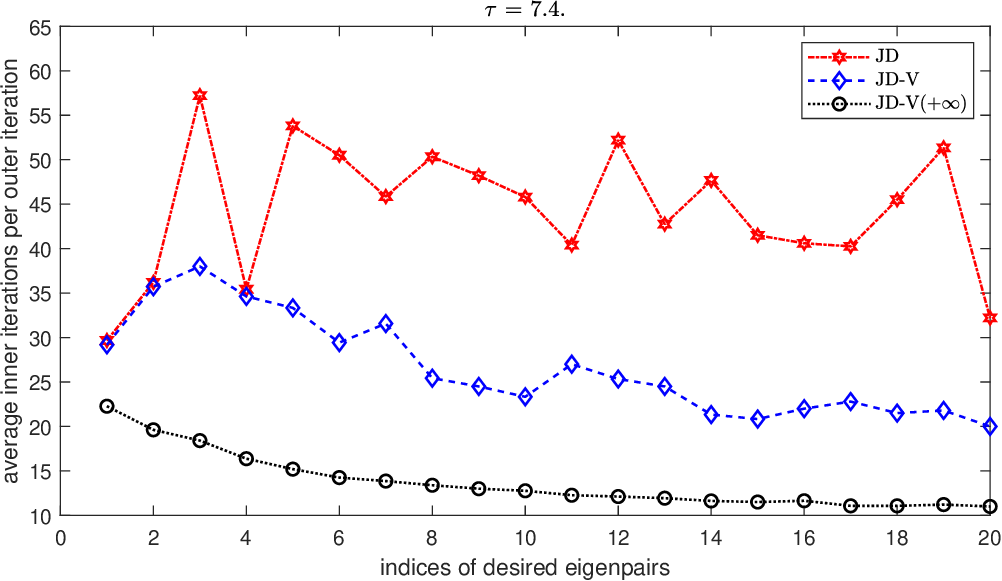}\hfill
	\includegraphics[width=0.48\textwidth]{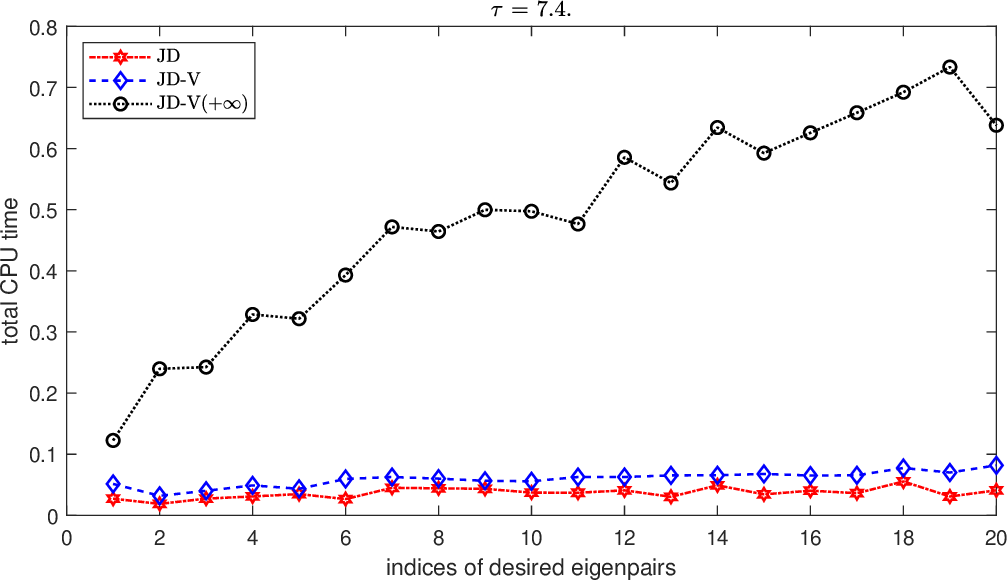}\\[0.2em]
	
	\includegraphics[width=0.48\textwidth]{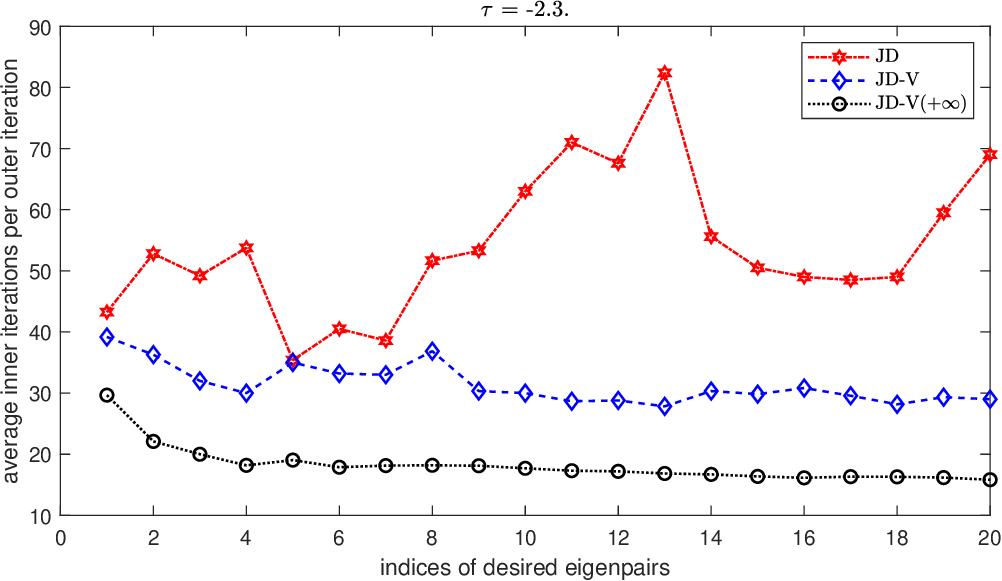}\hfill
	\includegraphics[width=0.48\textwidth]{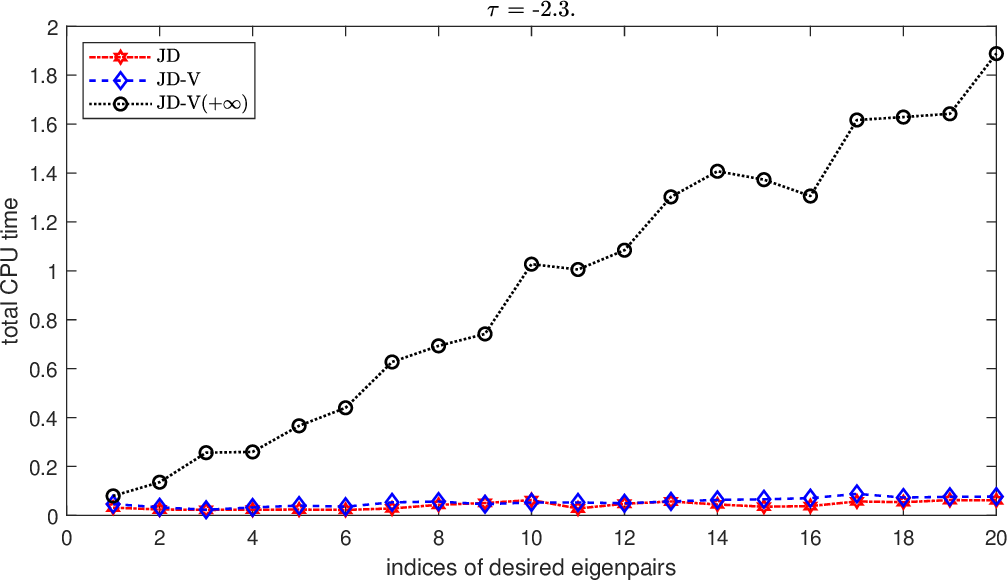}\\[0.2em]
	
	\includegraphics[width=0.48\textwidth]{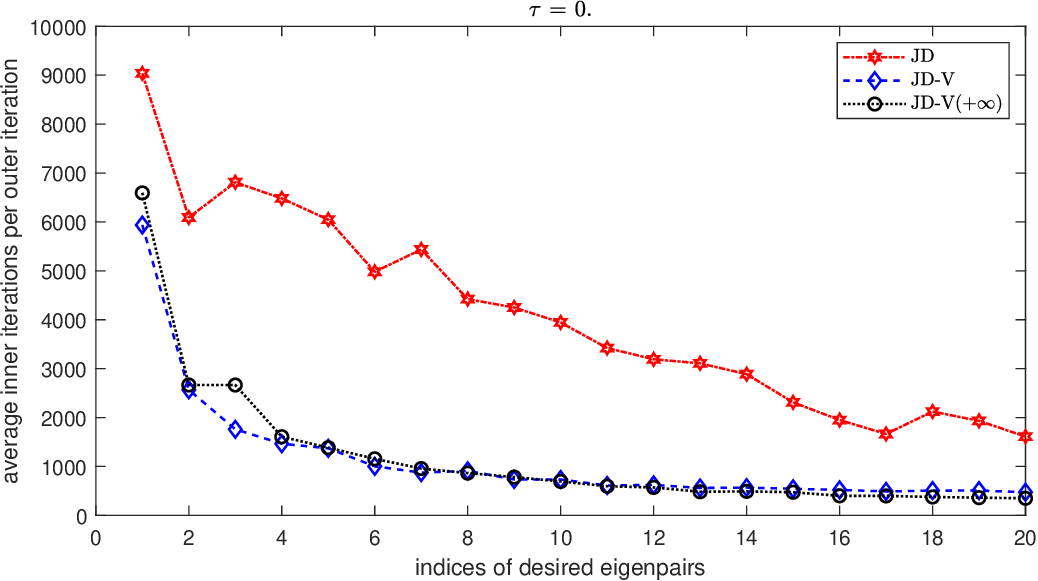}\hfill
	\includegraphics[width=0.47\textwidth]{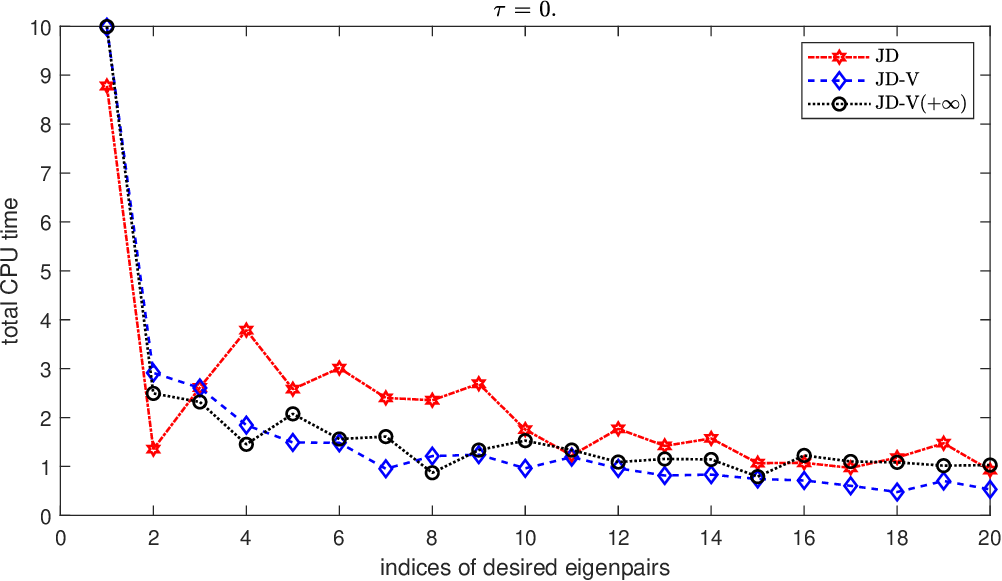} 
	
	\caption{Computing $20$ eigenpairs of ``barth'' using the unrestarted JD-type algorithms.}\label{fig1a}
\end{figure}
 
For each target, all three unrestarted algorithms successfully compute all 
the desired eigenpairs of $A$. Table~\ref{table1a} reports the 
computational results, where $i_{\mathrm{out}}$ and $i_{\mathrm{in}}$ 
denote the total numbers of outer and inner iterations required for convergence, respectively;  ``time'' represents the overall CPU time 
in seconds measured via the MATLAB timer; and boldface values indicate 
the optimal results. Figure~\ref{fig1a} depicts the average number of 
inner iterations per outer iteration (left) and total CPU time (right) 
each algorithm used to compute each desired eigenpair.
 
As shown in Table~\ref{table1a}, for each $\tau$, JD-V requires slightly 
or even considerably more outer iterations than JD, for the reasons 
detailed in Remark~\ref{remark9}.     
Nonetheless, it reduces the inner iterations of JD by 
23.60\%, 29.12\%, and 41.15\% when computing the largest, smallest, and 
interior eigenpairs, respectively, demonstrating that MINRES solves 
the correction equations in JD-V more efficiently than those JD, as corroborated by Figure~\ref{fig1a} (left).  
Consequently, JD-V consumes moderately more CPU time than JD for the first 
two relatively straightforward problems. 
This is primarily attributed to the growing overhead of the extraction 
phase during outer iterations; for this moderately sized matrix, 
computing the spectral decomposition of the projected matrix becomes 
increasingly expensive as the search subspace expands. 
However, for the most challenging problem ($\tau=0$), JD-V 
reduces the CPU time by 26.59\% compared to JD, confirming the efficacy of 
the proposed correction equations. 
  
We see from Figure~\ref{fig1a} that for most of the desired eigenpairs 
across all targets, JD-V($+\infty$) requires fewer, and occasionally 
substantially fewer, inner iterations on average per outer iteration 
than both JD and JD-V. 
However, it suffers from a significantly higher number of outer 
iterations, driven by the factors explained in Remark~\ref{remark3}. 
As a result, the overall efficiency of JD-V($+\infty$) remains considerably inferior to that of JD and JD-V in 
terms of both iteration counts and total CPU time.  

Furthermore, we observe that for the highly challenging interior eigenvalue problem ($\tau=0$), each algorithm requires substantially fewer outer iterations to compute the subsequent eigenpairs than to compute the first one, validating the effectiveness of the purgation technique introduced in section~\ref{subsec:8}. In addition, as illustrated in Figure~\ref{fig1a}, MINRES converges significantly faster on average for the second through last eigenpairs. This occurs because high-quality spectral information accumulates within the search subspace, which in turn leads to better-conditioned correction equations. As a result, all three algorithms consume substantially less CPU time computing the subsequent eigenpairs than the first one. Notably, for this problem, although JD-V uses more CPU time than JD to compute the first eigenpair due to more outer iterations, it ultimately achieves a lower total CPU time owing to the computational savings accrued during the computation of subsequent eigenpairs. 
This demonstrates that \emph{the advantages of JD-V over the standard JD become increasingly pronounced when multiple eigenpairs, rather than a single one, is computed}.
 
\begin{exper}\label{exper2}
Following the exact setup in Experiment~\ref{exper1}, 
compute the $20$ eigenpairs of $A$ using the thick-restart JD, JD-V, and JD-V($+\infty$) algorithms, with the search subspace dimension bounded by $k_{\max}=30$ and $k_{\min}=3$.	 
\end{exper} 

For each target, all three algorithms successfully compute the desired eigenpairs of $A$, with the numerical results reported in 
Table~\ref{table1b} and Figure~\ref{fig1b}.
 
\begin{table}[tbhp]
	\caption{Results of the thick-restart JD-V algorithms on ``barth'' with $\ell=20$.}\label{table1b}
	\begin{center}
		\begin{tabular}{cccccccccc} \toprule
			\multirow{2}{*}{$\tau$}
			&\multicolumn{3}{c}{JD}
			&\multicolumn{3}{c}{JD-V} 
			&\multicolumn{3}{c}{JD-V($+\infty$)} 	\\
			\cmidrule(lr){2-4} \cmidrule(lr){5-7}\cmidrule(lr){8-10}
			&$i_{\mathrm{out}}$ &$i_{\mathrm{in}}$ &time  
			&$i_{\mathrm{out}}$ &$i_{\mathrm{in}}$ &time  
			&$i_{\mathrm{out}}$ &$i_{\mathrm{in}}$ &time   \\   \midrule 
			\ 7.4  &125  &6269 &0.74   &\textbf{123} &\textbf{3639} &\textbf{0.72}    &548 &18066 &2.97 \\
			-2.3  &\textbf{117} &7092 &\textbf{0.79}   &132 &\textbf{ 4310} &\textbf{0.79}   &546 &19844 &3.17 \\
			0 	   &\textbf{119} &601040 &50.7   &211 &\textbf{305923} &\textbf{31.5}   &413 &1094717 &102 \\
			\bottomrule
		\end{tabular}
	\end{center}
\end{table}

\begin{figure}[tbp]
	\centering
	\includegraphics[width=0.48\textwidth]{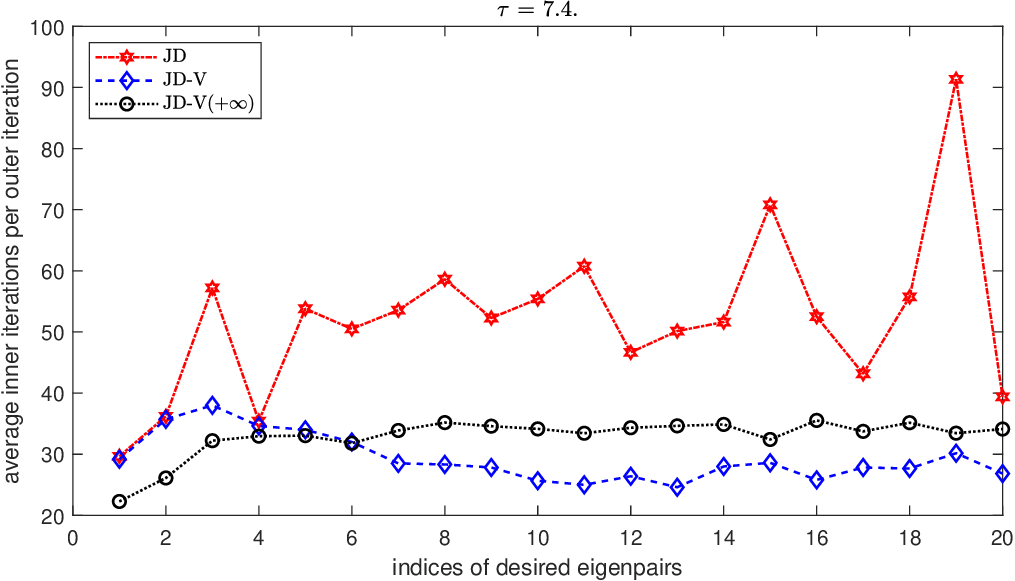}\hfill
	\includegraphics[width=0.48\textwidth]{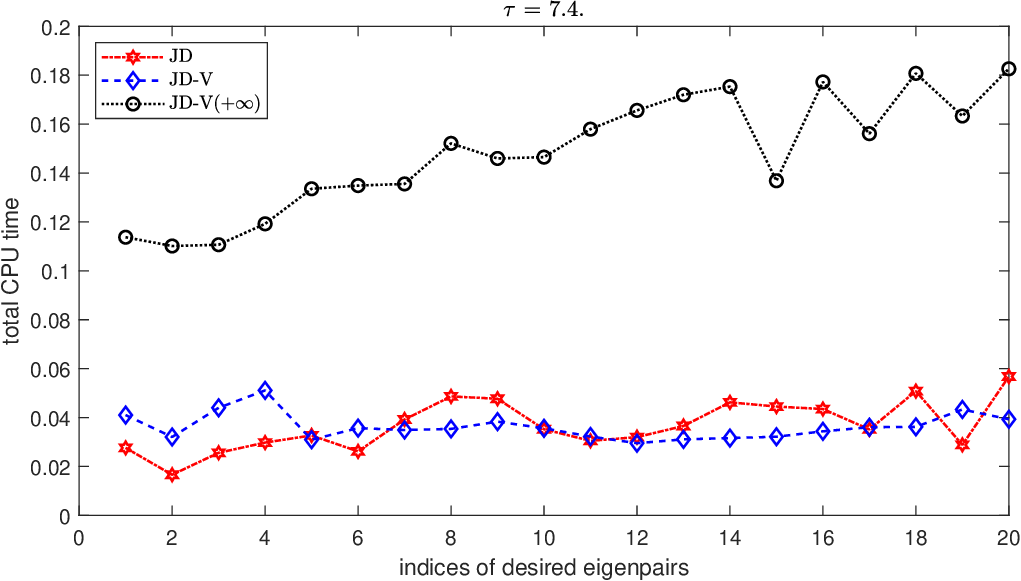}\\[0.2em]
	
	\includegraphics[width=0.48\textwidth]{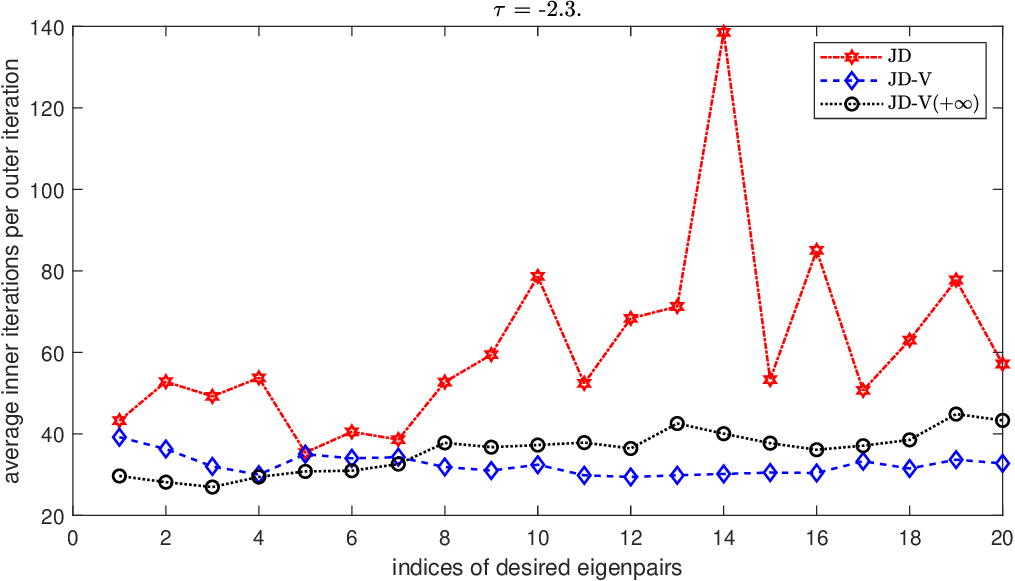}\hfill
	\includegraphics[width=0.48\textwidth]{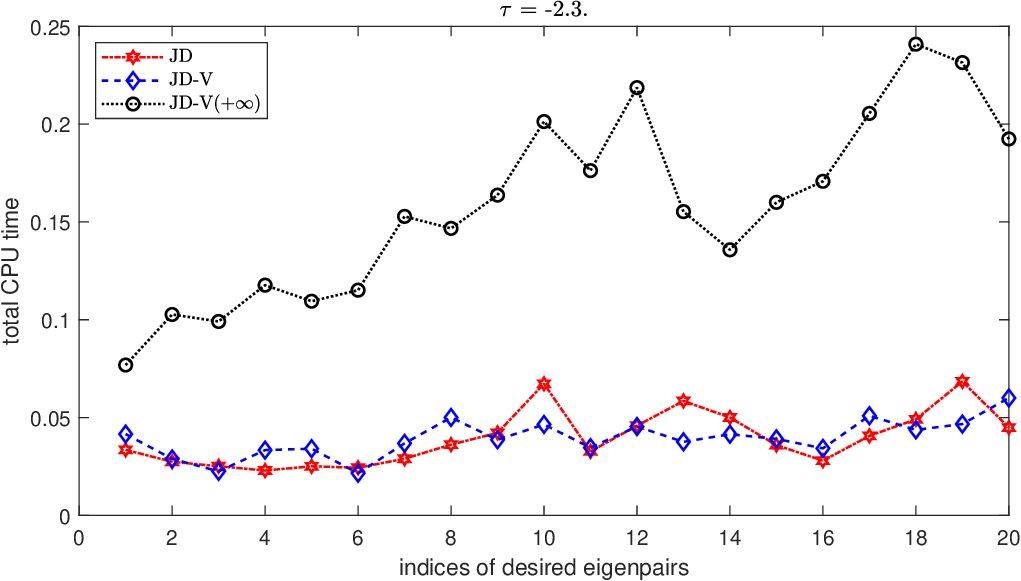}\\[0.2em]
	
	\includegraphics[width=0.48\textwidth]{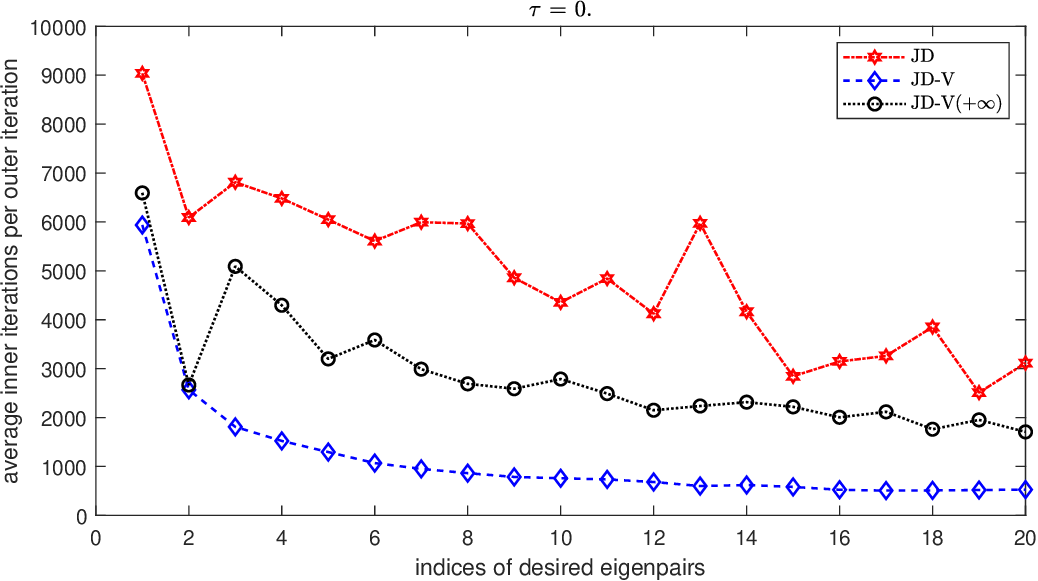}\hfill
	\includegraphics[width=0.47\textwidth]{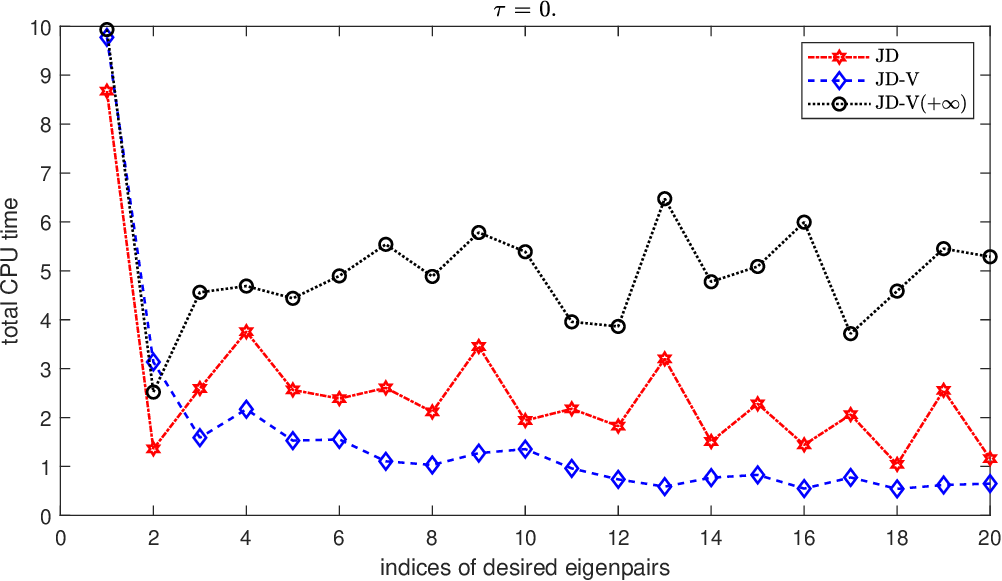} 
	
	\caption{Computing 20 eigenpairs of ``barth'' using the thick-restart JD-type algorithms.}\label{fig1b}
\end{figure} 

Comparing the results in Tables~\ref{table1a} and \ref{table1b}, we 
observe that for most of the test problems, the thick-restart JD and  JD-V($+\infty$) algorithms requires more and, in some cases, 
substantially more outer iterations than their unrestarted counterparts 
to achieve convergence, thereby necessitating more inner iterations  
and, in certain cases, longer CPU times. 
This is expected, as the standard thick restart strategy inevitably 
discards high-quality spectral information from the search subspace. 
In contrast, thick-restart JD-V requires a number of outer iterations comparable to that of its unrestarted variant across all problems, demonstrating the effectiveness of the adaptive thick-restart 
strategy described in section~\ref{subsec:4}. 
Notably, for the first two problems, although thick-restart JD-V 
uses slightly more inner iterations than its unrestarted counterpart, 
it consumes considerably less total CPU time. 
This highlights the non-negligible computational overhead associated with 
the extraction phase in the unrestarted algorithm for such moderately sized matrices. 
 
Among the three thick-restart JD‑type algorithms, we observe qualitatively similar performance to that described previously. 
In summary, JD-V outperforms JD: while it requires a comparable number of outer iterations for the first two problems and more for the third, 
it achieves a 39.23\%--49.10\% reduction in total inner iterations. 
This translates to comparable CPU times for the first two cases and a 
37.87\% savings in CPU time for the third. 
Ultimately, both thick-restart JD and JD‑V significantly outperform JD-V($+\infty$) in terms of both iteration counts and total CPU time.  
 
\begin{exper}
Compute the 20 eigenpairs of the eight matrices listed in 
Table~\ref{table2} for their respective targets $\tau$ using the 
{\em thick-restart} JD and JD-V algorithms, with 
$\varepsilon_{\mathrm{out}}=10^{-12}$ and all other parameters kept at 
their default values. 
\end{exper}

The desired eigenpairs for ``nemeth01'' and ``crack'', and for 
``spmsrtls'' and ``Kuu'', are the largest and smallest clustered ones, 
respectively. Those of the remaining four matrices correspond to 
interior clustered eigenvalues. 
Both algorithms successfully compute  
all target eigenpairs across all test matrices, and the computational results are summarized in Table~\ref{table2}.  
 
\begin{table}[tbhp]
 	\caption{Results for computing $20$ eigenpairs of eight matrices from \cite{davis2011university} with specified targets, where ``spa'' $=\mathrm{nnz}/n$ represents the average number of nonzeros 
 		per column, reflecting the sparsity.}\label{table2}
 	\begin{center}
 		\begin{tabular}{cccccccccc}
 			\toprule
 			\multirow{2}{*}{$A$}&\multirow{2}{*}{$n$}&\multirow{2}{*}{$spa$} &\multirow{2}{*}{$\tau$}
 			&\multicolumn{3}{c}{JD}
 			&\multicolumn{3}{c}{JD-V}\\
 			\cmidrule(lr){5-7} \cmidrule(lr){8-10}
 			&&&&$i_{\mathrm{out}}$ &$i_{\mathrm{in}}$ &time 
 			&$i_{\mathrm{out}}$ &$i_{\mathrm{in}}$ &time  \\\midrule
 			
 			nemeth01    &9506  &76.3  &5.4   &\textbf{95}  &336952  &107   &326 &\textbf{130256} &\textbf{59.2} \\
 			crack       &10240 &5.9   &6.4   &\textbf{108} &18702   &2.48  &116 &\textbf{8686}   &\textbf{1.77} \\
 			spmsrtls    &29995 &7.7   &\!-14.6\! &\textbf{123} &33033   &8.85  &128 &\textbf{14286}  &\textbf{5.81} \\
 			Kuu         &7102  &47.9  &0     &\textbf{74}  &21200   &3.47  &103 &\textbf{11958}  &\textbf{2.27} \\
 			\!wing\_nodal\! &10937 &13.8  &0     &\textbf{107} &1163064 &247   &214 &\textbf{595309} &\textbf{142}  \\
 			copter1     &17222 &12.3  &0     &\textbf{111} &607426  &122   &192 &\textbf{323165} &\textbf{75.5} \\
 			whitaker3   &9800  &5.9   &0     &\textbf{127} &656469  &73.1  &221 &\textbf{331260} &\textbf{46.0} \\
 			cti         &16840 &5.7   &0     &\textbf{113} &2064157 &328   &254 &\textbf{982183} &\textbf{212}  \\
 			
 			\bottomrule
 		\end{tabular}
 	\end{center}
\end{table}
 
As observed from Table~\ref{table2}, JD-V requires more, and in some cases 
substantially more, outer iterations than JD across all eight test 
matrices, for the reasons detailed in Remark~\ref{remark9}. 
Nevertheless, it consumes substantially fewer inner iterations and 
considerably less CPU time. 
Specifically, when computing the extreme eigenpairs of ``nemeth01'',  ``spmsrtls'', and ``Kuu'',   
JD-V reduces total inner iterations by 43.59\%--61.34\% and the 
overall CPU time by 28.66\%--44.97\% compared to JD. 
For the more challenging clustered interior eigenproblems of the remaining four matrices (``wing\_nodal'', ``copter1'', ``whitaker3'', and ``cti''), where JD requires substantially more inner iterations,  JD‑V achieves reductions of 46.80\%--52.42\% in inner iterations and 35.31\%--42.59\% in CPU time, despite requiring 73\%--125\% more outer iterations. 
Overall, JD‑V clearly outperforms JD across all eight test problems in terms of both total inner iterations and CPU time. 
 
In particular, we observe that for the four highly sparse matrices ``crack'', ``spmsrtls'', ``whitaker3'' and ``cti'', the percentage reduction in CPU time achieved by JD-V relative to JD (28.66\%, 34.35\%, 37.16\%, and 35.31\%, respectively) is considerably smaller than the corresponding reduction in total inner iterations (53.56\%, 56.75\%, 49.54\%, and 52.42\%). 
For these matrices, the computational cost per MINRES inner iteration when solving the correction equation \eqref{correctionm} (or the deflated \eqref{correctionfinal}) in JD-V, which is dominated by the matrix-vector product involving the coefficient matrix, can be much higher than that of MINRES solving the standard correction equation \eqref{correction} (or its deflated variant) in JD. This demonstrates that the new correction equations in JD-V considerably improve inner-iteration efficiency, although for highly sparse matrices, the overall CPU time benefit may be partially offset by the increased cost per inner iteration. 
 
\section{Conclusions}\label{sec:6} 
For large-scale Hermitian eigenvalue problems, we have proposed a novel JD 
variant, designated as JD-V, based on a new correction equation. 
Its solution is shown to be nearly as effective as, or marginally less so 
than, that to the standard correction equation in the conventional JD 
method. We have carried out rigorous convergence analyses of MINRES for solving these two equations, showing that the inner iterations of JD-V 
can converge significantly faster that those of JD. 
We have developed a thick-restart JD-V algorithm incorporating deflation and purgation for computing several eigenpairs of a large-scale Hermitian matrix.
Numerical experiments confirm that for highly clustered eigenpairs, 
despite an occasional, marginal delay in outer convergence, JD-V 
substantially outperforms standard JD in terms of both total inner 
iteration counts and overall CPU time.
 
\section*{Declarations}

The author declare that she has no financial interests, 
and she read and approved the final manuscript. 
All numerical experiments can be reproduced using the MATLAB
implementations of the JD-V algorithm, which is publicly
available at \url{https://github.com/JinzhiHuang/JD-VH}.

%The algorithmic MATLAB code is available upon reasonable 
%request from the author.

\bibliographystyle{siamplain}
\bibliography{ipjdeigs_ref}
\end{document}